\documentclass{amsart}

\usepackage{amsfonts}
\usepackage{amsmath}
\usepackage{amssymb}
\usepackage{amsthm}

\makeatletter
\renewcommand\@seccntformat[1]{\csname the#1\endcsname.\quad}

\newtheorem{theorem}{Theorem}[section]
\newtheorem{thm}{Theorem}[section]

\newtheorem{coro}[theorem]{Corollary}

\newtheorem{defn}[theorem]{Definition}
\newtheorem{exam}[theorem]{Example}

\newtheorem{lem}[theorem]{Lemma}

\newtheorem{rem}[theorem]{Remark}

\newcommand{\norm}[3]{\ensuremath{\left\Vert#1\right\Vert_{#2}^{#3}}}

\DeclareMathOperator{\dist}{dist}

\begin{document}

\author[M. Ciesielski \& G. Lewicki]{Maciej Ciesielski$^*$ and Grzegorz Lewicki}

\thanks{$^*$This research is supported by the grant 04/43/DSPB/0086 from Polish Ministry of Science and Higher Education.}

\title{Best approximation properties in spaces of measurable functions}

\begin{abstract}
We research proximinality of $\mu$-sequentially compact sets and $\mu$-compact sets in measurable function spaces. Next we show a correspondence between the Kadec-Klee property for convergence in measure and $\mu$-compactness of the sets in Banach function spaces. Also the property $S$ is investigated in Fr\'echet spaces and employed to provide the Kadec-Klee property for local convergence in measure. We discuss complete criteria for continuity of metric projection in Fr\'echet spaces with respect to the Hausdorff distance. Finally, we present the necessary and sufficient condition for continuous metric selection onto a one-dimensional subspace in sequence Lorentz spaces $d(w,1)$.       
\end{abstract}

\maketitle

\bigskip\ 

{\small \underline{2000 Mathematics Subjects Classification: 41A65, 46E30, 46A40 }\hspace{1.5cm}\
\ \ \quad\ \quad . }\smallskip\ 

{\small \underline{Key Words and Phrases:}\hspace{0.15in} Fr\'echet spaces, Banach function spaces, Lorentz spaces, Kadec-Klee property, metric selection, extreme points, $\mu$-compactness.}

\bigskip\ \ 

\section{Introduction}
The geometrical structure of Banach spaces has been investigated  tremendously and applied to approximation theory by many authors \cite{CheHeHudz,CiesKamPluc,HKL,Hu-Ku,kurc}. Since the deep motivation of study of the geometry of Banach spaces has been developed, during the decades selected global properties corresponding to a metric have been evaluated \cite{Cal,CKKP,hk1,Kam_extrem,Kam,KMGam}. Namely, the monotonicity and rotundity properties of Banach spaces (for example rotundity, uniform rotundity, strict monotonicity and uniform monotonicity) are crucial key in investigation of existence and uniqueness at the best approximation problems \cite{CheHeHudz,Hu-Ku,kurc}. If the global structure of Banach space disappoints, then the natural question appears of researching the applicable local structure of Banach space in approximation problems. Namely, local approach of rotundity, monotonicity and Kadec Klee property for global and local convergence in measure, resp. (for example extreme points, points of lower and upper monotonicity, $H_g$ and $H_l$ points) with application to the best approximation problems has been evolved recently by \cite{CiesKamPluc,CieKolPan,CieKolPlu}. In view of the previous results, a natural expectation that the metric structure of Fr\'echet spaces plays an analogous rule as Banach structure in approximation theory is researched in this paper.   
The next point of our interest in this paper is an existence of continuous metric selection onto a one-dimensional subspace of the sequence Lorentz spaces. It is worth to noticing that the complete condition under which there exists a continuous metric selection in $C_0(T)$ and $L^1$ was researched in \cite{Deutsch}. In the spirit of these results we investigate criteria for continuity of the metric selection onto a one-dimensional subspace of the Lorentz spaces $d(w,1)$. For more information concerning various concepts of continuity of the metric projection operator the reader is referred to \cite{Brown1} - \cite{Brown5} and \cite{Wulbert}. The last point of our consideration in this paper is devoted to the property $S$ that was established unexplicitly for the first time in \cite{KPS} and applied to provide the Kadec-Klee property for convergence in measure in Lorentz spaces $\Lambda_\phi$.
 
In section 2 we recall the necessary terminology.

The section 3 is devoted to investigation of  proximinality of $\mu$-sequentially compact sets and $\mu$-compact sets in measurable function spaces. We present examples and properties of $\mu$-sequentially compact and $\mu$-compact sets. We positively answer the essential question under which criteria in a Fr\'echet space $X$, equipped with a $F$-norm, $\mu$-sequentially compactness yields proximinality of closed subset of $X$. We also discuss some special proximinal sets and Chebyshev sets in Fr\'echet spaces. Finally, in this section we apply Kedec-Klee property for global convergence in measure to establish proximinality of $\mu$-compact sets in Banach function spaces. We also present a class of Lorentz spaces $\Lambda_\phi$ and $\Gamma_{p,w}$ which possess a property that any nonempty $\mu$-compact subset in these spaces is proximinal. 

In the next section 4 we characterize a continuity of metric projection operator and property $(S)$ in Banach function spaces and Fr\'echet spaces. Namely, we show examples of Lorentz spaces $\Lambda_\phi$ and $\Gamma_{p,w}$ and also some specific Fr\'echet spaces which satisfy property $S$. We also prove that the direct sum of Fr\'echet spaces equipped with $F$-norm satisfying property $S$ is a Fr\'echet space equipped with $F$-norm and has Kadec-Klee property for local convergence in measure.
Finally, we establish the necessary and sufficient criteria for continuity of the metric projection with respect to the Hausdorff distance in Fr\'echet function space with the Kadec-Klee property for local convergence in measure.  It is worth noticing that in our considerations we do not restrict ourselves to the case of Banach spaces but we also consider Fr\'echet spaces not necessarily locally convex. Observe
that in abstract approximation theory there are very few papers concerning Fr\'echet spaces. Mainly the case of Banach spaces is considered (see for example taken in a random way volume of Journal of Approximation Theory or Constructive Approximation).

The last section 5 is devoted to the characterization of the metric selection for the metric projection onto a one-dimensional subspace in sequence Lorentz spaces $d_{(w,1)},$ which surprisingly is a highly non-trivial problem (compare with \cite{Deutsch}, Th. 6.3 and Cor. 6.6). We present full criteria for continuity of the metric selection onto a one-dimensional subspace in sequence Lorentz space $d_{(w,1)}$.  We also discuss some examples of the metric projection onto a one-dimensional subspace in sequence Lorentz spaces which does not admit a continuous metric selection and also admits a continuous metric selection. 

\section{Preliminaries}
Let $\mathbb{R}$, $\mathbb{R}^+$ and $\mathbb{N}$ be the sets of reals, nonnegative reals  and positive integers, respectively. Denote as usual by $S_X$ (resp. 
$B_X$) the unit sphere (resp. the closed unit ball) of a Banach
space $(X,\left\Vert \cdot \right\Vert _{X}).$  A point $x\in{B_X}$ is called an {\it extreme point} of $B_X$ if for any $y,z\in{B_X}$ such that $y+z=2x$ we have $y=z$. A Banach space $X$ is said to be {\it strictly convex} if any element $x\in{S_X}$ is an extreme point of $B_X$. Define by $(T,\Sigma,\mu)$ a measure space and by $L^{0}(T)$ the set of all (equivalence classes of) extended real valued $\mu$ measurable functions on $T$. For simplicity we use the short notation $L^0=L^0([0,\alpha))$ with the Lebesgue measure $m$ on $[0,\alpha)$, where $\alpha =1$ or $\alpha =\infty$. A Banach lattice $(E,\Vert \cdot \Vert _{E})$ is called a \textit{Banach function space} (or a \textit{K\"othe space}) if it is a sublattice of $L^{0}$ satisfying the following conditions
\begin{itemize}
\item[(1)] If $x\in L^0$, $y\in E$ and $|x|\leq|y|$ m-a.e., then $x\in E$ and $%
\|x\|_E\leq\|y\|_E$.
\item[(2)] There exists a strictly positive $x\in E$.
\end{itemize}
By $E^{+}$ we denote the positive cone of $E$, i.e. $E^{+}={\{x \in E:x \ge 0,m\textnormal{-a.e.}\}}$. We denote $A^{c}=[0,\alpha)\backslash A$ for any measurable set $A$. A space $E$ has the \textit{Fatou property} if for any $\left( x_{n}\right)\subset{}E^+$, $\sup_{n\in \mathbb{N}}\Vert x_{n}\Vert
_{E}<\infty$ and $x_{n}\uparrow x\in L^{0}$, then $x\in E$ and $\Vert x_{n}\Vert _{E}\uparrow\Vert x\Vert_{E}$. 
A space $E$ has the semi-\textit{Fatou property} if conditions $0\leqslant
x_{n}\uparrow x\in E$ with $x_{n}\in E$ imply $\Vert x_{n}\Vert _{E}\uparrow
\Vert x\Vert _{E}$. For any function $x\in L^{0}$ we define its \textit{distribution function} by 
\begin{equation*}
d_{x}(\lambda) =m\left\{ s\in [ 0,\alpha) :\left\vert x\left(s\right) \right\vert >\lambda \right\},\qquad\lambda \geq 0,
\end{equation*}
and its \textit{decreasing rearrangement} by 
\begin{equation*}
x^{\ast }\left( t\right) =\inf \left\{ \lambda >0:d_{x}\left( \lambda
\right) \leq t\right\}, \text{ \ \ } t\geq 0.
\end{equation*}
Given $x\in L^{0}$ we denote the \textit{maximal function} of $x^{\ast }$ by 
\begin{equation*}
x^{\ast \ast }(t)=\frac{1}{t}\int_{0}^{t}x^{\ast }(s)ds.
\end{equation*}%
Two functions $x,y\in{L^0}$ are said to be \textit{equimeasurable} (shortly $x\sim y$) if $d_x=d_y$.
A Banach function space $(E,\Vert \cdot \Vert_{E}) $ is called \textit{rearrangement invariant} (r.i. for short) or \textit{symmetric} if whenever $%
x\in L^{0}$ and $y\in E$ with $x \sim y,$ then $x\in E$ and $\Vert x\Vert
_{E}=\Vert y\Vert _{E}$. For more properties of $d_{x}$, $x^{\ast }$ and $x^{\ast \ast }$ see \cite{BS, KPS}. 
% Lorentz space $\Gamma_{p,w}$
Let $0<p<\infty $ and $w\in L^{0}$ be a nonnegative weight function, the
Lorentz space $\Gamma _{p,w}$ is a subspace of $L^{0}$ such that 
\begin{equation*}
\Vert x\Vert _{\Gamma _{p,w}}:=\left( \int_{0}^{\alpha }x^{\ast \ast
	p}(t)w(t)dt\right) ^{1/p}<\infty .
\end{equation*}
Additionally, we assume that $w$ is from
class $D_{p}$, i.e. 
\begin{equation*}
W(s):=\int_{0}^{s}w(t)dt<\infty \mathnormal{\ \ \ }\text{\textnormal{and}}%
\mathnormal{\ \ }W_{p}(s):=s^{p}\int_{s}^{\alpha }t^{-p}w(t)dt<\infty
\end{equation*}%
for all $0<s\leq 1 $ if $\alpha =1 $ and for all $0<s<\infty $
otherwise. These two conditions guarantees that Lorentz space $\Gamma _{p,w}$ is nontrivial. It is well known that $\left( \Gamma _{p,w},\Vert \cdot \Vert
_{\Gamma _{p,w}}\right) $ is a r.i. quasi-Banach function space with the Fatou
property. It was proved in \cite{KMGam} that in the case when $\alpha =\infty $
the space $\Gamma _{p,w}$ has order continuous norm if and only if $%
\int_{0}^{\infty }w\left( t\right) dt=\infty .$ The spaces $\Gamma _{p,w}$ were introduced by A.P. Calder\'{o}n in \cite{Cal} in a similar way as the classical Lorentz spaces $\Lambda_{p,w}$ that is a subspace of $L^0$ with
\begin{equation*}
\left\Vert x\right\Vert _{\Lambda_{p,w}}=\left( \int_{0}^{\alpha}(x^{\ast }(t))^{p}w(t)dt\right) ^{1/p}<\infty,
\end{equation*}%
where $p\geq 1$ and the weight function $w$ is nonnegative and nonincreasing (see \cite{Loren}). The space $\Gamma _{p,w}$ is an interpolation space between $L^{1}$ and $L^{\infty }$ yielded by the Lions-Peetre $K$-method \cite{BS,KPS}. Clearly, $\Gamma _{p,w}\subset \Lambda _{p,w}.$ The opposite
inclusion $\Lambda _{p,w}\subset \Gamma _{p,w}$ is satisfied if and only if $w\in B_{p}$ (see \cite{KMGam}). It is worth mentioning that the spaces $\Gamma _{p,w}$ and $\Lambda _{p,w}$ are
also connected by Sawyer's result (Theorem 1 in \cite{Sawy}; see also \cite{Step}), which states that the K\"{o}the dual of $\Lambda _{p,w}$, for $1<p<\infty $ and $\int_{0}^{\infty }w(t)dt=\infty $, coincides with the
space $\Gamma _{p^{\prime },\widetilde{w}}$, where $1/p+1/p^{\prime }=1$ and 
$\widetilde{w}(t)=\left( t/\int_{0}^{t}w(s)ds\right) ^{p^{\prime }}w(t)$. 
	
Let $(X,\tau)$ be a topological vector space, where $X$ is a vector space and $\tau$	is topology. A topological vector space $(X,\tau)$ is said to be a {\it Fr\'echet space} if its topology $\tau$ is induced by a translation invariant metric $d$, i.e. $d(x+z,y+z)=d(x,y)$ for any $x,y,z\in{X}$, and also $(X,d)$ is complete. Unless we say otherwise we consider a Fr\'echet space $X$ with a topology $\tau$ induced by a $F$-norm, i.e. a mapping $\norm{\cdot}{}{}:X\rightarrow\mathbb{R}^+$ satisfying the triangle inequality and the following conditions; $(i)$ $\norm{x}{}{}=0\Leftrightarrow{x=0}$, and $(ii)$ $\norm{x}{}{}=\norm{-x}{}{}$.

% $Kadec-Klee$ porperties
Let  $ X \subset L^0(T)$ be a Fr\'echet function space equipped with an F-norm $ \| \cdot \|$. A point $x\in{X}$ is said to be an $H_g$ \textit{point} in $X$ if for any $(x_n)\subset{X}$ such that $x_n\rightarrow{x}$ globally in measure and $\norm{x_n}{}{}\rightarrow\norm{x}{}{}$, we have $\left\Vert x_n-x\right\Vert\rightarrow{0}$. A point $x\in{X}$ is said to be an $H_l$ \textit{point} in $X$ if for any $(x_n)\subset{X}$ such that $x_n\rightarrow{x}$ locally in measure and $\norm{x_n}{}{}\rightarrow\norm{x}{}{}$, we have $\left\Vert x_n-x\right\Vert\rightarrow{0}$. We say that the space $X$ has the \textit{Kadec-Klee property for global convergence in measure} (\textit{Kadec-Klee property for local convergence in measure}) if each $x\in{X}$ is an $H_g$ point (an $H_l$ point) in $X$.

Let $ X $ be a Fr\'echet space with a $F$-norm $\norm{\cdot}{}{}$ and let $ Y \subset X$ be a nonempty subset. For $x \in X$ define 
$$
P_Y(x) = \{ y \in Y: \| x - y\| = dist(x,Y)\}.
$$ 
Any $y \in P_Y(x)$ is called a best approximant in $Y$ to $x$ and the mapping $ x \rightarrow P_Y(x)$ is called {\it the metric projection.} A nonempty set $ Y \subset X$ is called {\it proximinal} if $ P_Y(x) \neq \emptyset$ for any $ x \in X.$ A nonempty set $Y$ is said to be a {\it Chebyshev set} if it is proximinal and $ P_Y(x)$ is a singleton for any $ x \in X.$ A continuous mapping $S : X \rightarrow Y$ is called a {\it continuous metric selection} if $Sx \in P_Y(x)$ for any $ x \in X.$

\section{Proximinality in spaces of measurable functions}
We start with the necessary notion. 
\begin{defn}
\label{mucompact}
Let $(T, \Sigma, \mu)$ be a measure space and let for $ t \in T$ 
$$ 
Z_t = \mathbb{R} \cup \{ - \infty, \infty \}.
$$
Assume that $ Z = \Pi_{t\in T} Z_t$ is equipped with the Tychonoff topology, which will be denoted by $ \tau.$
A set $ C \subset Z$ is called $\mu$-sequentially compact if for any sequence $ \{ c_n\} \subset C$ there exists a subsesquence $ \{ c_{n_k}\}$ and $ c \in Z$ such that 
$c_{n_k}(t) \rightarrow c(t)$ $ \mu$-a.e.. 
\end{defn}
\begin{rem}
\label{counting}
If $ T = \mathbb{N}, \Sigma = 2^{\mathbb{N}}$ and $ \mu$ is the counting measure, then by diagonal argument any set $C \subset Z$ is $ \mu$-sequentially compact.
\end{rem}
Now we present some examples of $ \mu$-sequentially compact sets. First we investigate $\mu$-sequentially compactness of a set of all increasing functions on $\mathbb{R}$.
\begin{thm}
\label{increasing}
Let $ T = \mathbb{R},$ $ \mu$ be a nonatomic measure on $T$ and let $\Sigma$ denote the $\sigma$-algebra of Borel subsets of $T.$ Let 
$$
C = \{ f: \mathbb{R} \rightarrow \mathbb{R}: f \hbox{ is increasing} \}.
$$ 
Then $C$ is $\mu$-sequentially compact.
\end{thm} 
\begin{proof}
Fix a sequnece $ \{ c_n \} \subset C.$ Let $ A_n = cl^{\tau}\{c_k: k \geq n\}.$ Since $ Z$ is $\tau$-compact,
$$
A = \bigcap_{n =1}^{\infty} A_n \neq \emptyset.
$$
Fix $ c \in A.$ First assume that $ |c(t)| < + \infty $ for any $ t \in \mathbb{R}.$ Since $ c_n$ are increasing it is easy to see that $c$ is also increasing.
Since $ c$ is increasing, it has at most countable number points of discontinuity. Let us denote this set by $ D_o.$ Since $ \mu$ is a nonatomic, Borel maesure,$\mu(D_o) = 0.$ Now fix a countable set  $ E= \{t_1,...\} \subset F= \mathbb{R} \setminus D_o$ dense in $F.$ Now we find a subsequence
$ \{n_k\} $ such that $ c_{n_k}(t_i) \rightarrow c(t_i)$ for any $i \in \mathbb{N}.$ To do that, set 
$$ 
U_k =\{ g \in A_k: |g(t_i) - c(t_i)| < \frac{1}{k} \hbox{ for } i=1,...,k \}.
$$  
By definition of $c,$ $\tau$ and $ A $ we can select a stricly increasing sequence $ \{ n_k \}$ such that $ c_{n_k} \in \{c_l: l > n_{k-1}\}$ for any $ k \in \mathbb{N}.$ Clearly, $ c_{n_k}(t_i) \rightarrow c(t_i)$ for any $i\in \mathbb{N}.$
\newline
Now we show that $c_{n_k}(t) \rightarrow c(t) \in \mathbb{R}$ for any $ t \in F.$  Assume to the contrary that there exists $ t \in F$ such that $c_{n_k}(t)$ does not converge to $c(t).$ Hence without loss of generality, passing to a subsequence if necessary, we can assume that there exists $ d>0$ such that $ |c_{n_k}(t) - c(t)| >d$ for any $ k \in \mathbb{N}.$ First assume that for infinite number of $k$
$c_{n_k}(t) - c(t) >d.$ Then for any $ s> t$
$$
c_{n_k}(s) - c(t) \geq c_{n_k}(t) - c(t) > d.
$$ 
Since $ c$ is continuous at $t$ select $ t_i \in E$ such that $ t < t_i$ and $ c(t_i) - c(t) < \frac{1}{3d}.$ 
Since for $ k \geq k_o $ $ |c_{n_k}(t_i) -c(t_i)| < \frac{1}{3d},$ we have for $ k \geq k_o,$
$$
0< d <c_{n_k}(t) - c(t) \leq c_{n_k}(t_i) - c(t) 
$$
$$
\leq |c_{n_k}(t_i) - c(t_i)| + |c(t_i) - c(t)| < \frac{2}{3d} < d,
$$
which is a contradiction.
\newline
If for infinite number of $k$ $c(t)-c_{n_k}(t)  >d,$ reasoning in the same way, we get a contradiction.
\newline
Now assume that there exists $ s_o \in F$ such that $ |c(s_o)| = +\infty.$ Without loss of generality we can assume that $ c(s_o) = +\infty.$ Let 
$$
s_1 = \inf \{ t \in \mathbb{R}: c(t) = +\infty \}
$$
$(s_1 = -\infty $ if for any $s\in\mathbb{R}$ $c(s)= +\infty ).$ Since $ c_{n_k}$ is increasing for any $k\in\mathbb{N}$, $c(s) = + \infty$ for any $ s > s_1.$ If $ s_1 \in \mathbb{R}, $ and $c(u) = -\infty $ for some $ u \leq s_1,$ then put  
$$
u_1 = \sup \{ t \in \mathbb{R}: c(t) = -\infty \}.
$$
If $ u_1 = s_1, $ then $c_{n_k}(t) \rightarrow +\infty $ for any $ t > s_1$ and $c_{n_k}(t) \rightarrow -\infty $ for any $ t < s_1.$ If $ u_1 < s_1$ then for any $ t \in (u_1,s_1) $ $ |c(t)| < +\infty.$ Reasoning as in the first part of the proof we get our claim. 
\end{proof}
\begin{thm}
\label{convex}
Let $ T = \mathbb{R},$ $ \mu$ be a nonatomic measure on $T$ and let $ \Sigma$ denote the $\sigma$-algebra of Borel subsets of $T.$ Let 
$$
C = \{ f: \mathbb{R} \rightarrow \mathbb{R}: f \hbox{ is convex} \}.
$$ 
Then $C$ is $\mu$-sequentially compact.
\end{thm}
\begin{proof}
Fix a sequence $ \{c_n \} \subset C.$ If there exists a subsequence $ \{n_k\} $  such that $ c_{n_k}$ is increasing for any $k\in\mathbb{N}$ or $c_{n_k}$ is decreasing for any $k\in\mathbb{N}$ then applying Theorem \ref{increasing} we get our claim. In the opposite case, select for any $ n\in\mathbb{N}$, $ a_n \leq b_n$ such that $ c_n$ attains its global minimum at $ t \in [a_n,b_n].$ Let 
$$
a = \inf \{ a_n : n \in \mathbb{N}\}.
$$
Without loss of generality, passing to a convergent subsequence, if necessary we can assume that $ a_n \rightarrow a.$ 
If $ a = -\infty ,$ then for any $ l \in \mathbb{N}$ there exists $ n_o(l) \geq l$ such that $ c_n|_{(-l, +\infty)}$ are increasing for $n \geq n_o(l).$  By Theorem \ref{increasing} and the diagonal argument we can select a subsequence 
$\{n_k\} $ such that $c_{n_k}(t) \rightarrow c(t)$ $ \mu$-a.e.. If $ a > -\infty$ then for any $ l \in \mathbb{N}$ there exists $ n_o(l) \geq l$ such that $ c_n|_{(a+\frac{1}{l}, +\infty)}$ are increasing for $n \geq n_o(l)$ and $ c_n|_{(-\infty, a-\frac{1}{l})}$ are decreasing. Reasoning as in the proof of Theorem \ref{increasing} and applying the diagonal argument we can select a subsequence 
$\{n_k\} $ such that $c_{n_k}(t) \rightarrow c(t)$ $ \mu$-a.e.. The proof is complete.
\end{proof}
Now we prove important for further applications 
\begin{thm}
\label{partition}
Let $ T = \mathbb{R}$ and let $ \mu$ a nonatomic measure on $T$ and let $ \Sigma$ the $\sigma$-algebra of Borel subsets of $T.$ Fix $ \{ t_n\}_{n \in \mathbb{Z}},$ $ t_n < t_{n+1},$ a partition of $ T.$
Assume that $ C \subset Z$ is such that for any $i \in \mathbb{N}$ 
$$ 
C|_{(t_i, t_{i+1})}= \{ c|_{(t_i, t_{i+1})}:c \in C\} 
$$ is $\mu$-sequentially compact. Then $C$ is $\mu$-sequentially compact.
\end{thm}
\begin{proof}
By Def. \ref{mucompact}, applying the diagonal argument we select a a subsequence $ \{ n_k \} $ and $ A_i \subset (t_i, t_{i+1}),$ $ \mu (A_i) =0,$ such that 
$ c_{n_k}(t) \rightarrow c(t)$ for any $t \in \mathbb{R} \setminus \bigcup_{i\in \mathbb{Z}} A_i.$ Since $ \mu \left( \bigcup_{i\in \mathbb{Z}} A_i \right) =0,$ we get our claim.
\end{proof}
\begin{rem}
It is clear that in Theorem \ref{increasing} and Theorem \ref{convex} the set of real numbers $T$ can be replaced by any open interval $(a,b)$ where $ a, b \in \mathbb{R} \cup \{ +\infty, -\infty \}.$ 
\end{rem}
Now we present other possible constructions of $\mu$-sequentially compact sets.
\begin{thm}
\label{constructions}
Let $(T, \Sigma, \mu)$ be a measure space. 
\begin{itemize}
\item[$(a)$] If $C$ is $\mu$-sequentially compact then any nonempty set $ D \subset C$ is $\mu$-sequentially compact.
\item[$(b)$] If $ C_i$  for $ i=1,..n$ are $\mu$-sequentially compact, then $C = \bigcup_{j=1}^{n}C_i$ is $\mu$-sequentially compact.
\item[$(c)$] Let $C_1$ and $C_2$ be $\mu$-sequentially compact subsets such that $C_i\subset\Pi_{t\in T} Z_t$ for $i=1,2$, where  $Z_t = [f(t), +\infty)$ and $f(t) >0$ for $t \in T.$ Then $C_1C_2 = \{ c_1c_2 : c_1 \in C_1, c_2 \in C_2 \}$ and $ aC_1 + bC_2$ with $ a, b\geq 0$ are  $\mu$-sequentially compact sets.
\item[$(d)$] If $C_1$ and $C_2$ are $\mu$-sequentially compact, then $ C_1\vee{C_2} = \{c_1\vee{}c_2: c_1 \in C_1, c_2 \in C_2 \}$ and $ C_1\wedge{C_2} = \{ c_1\wedge{}c_2: c_1 \in C_1, c_2 \in C_2 \}$ are $\mu$-sequentially compact sets, where $(c_1\vee{}c_2)(t) = \max(c_1(t),c_2(t))$ and $(c_1\wedge{}c_2)(t) = \min(c_1(t),c_2(t)).$
\item[$(e)$] Let $ f : T \rightarrow \mathbb{R}_+$ be a function. If $C$ is $\mu$-sequentially compact, then $ P_f(C)$ is also $\mu$-sequentially compact, where for $ c \in C$ and $t \in T,$ 
$ P_f(c)(t) = c(t)$ if $ |c(t)| \leq f(t)$ and $ P_f(c)(t)= sgn(c(t))f(t)$ in the opposite case.
\end{itemize}
\end{thm}

\begin{proof}
The proofs of conditions from $(a)$ to $(d)$ follows immediately from definition of the limit, Def. \ref{mucompact} and the fact that countable union of sets of measure zero has also measure zero.
To prove (e) fix a sequence $\{d_n\} \subset P_f(C).$ Let $ d_n = P_f(c_n), $ where $c_n \in C$ for any $n \in \mathbb{N}.$ Since $ C$ is $\mu$-sequentially compact, we can select a subsequence 
$ \{ c_{n_k}\}$ such that $ c_{n_k} \rightarrow c(t) \in \mathbb{R} \cup \{ -\infty, +\infty \}$ $ \mu$-a.e. If $ |c(t)| > f(t),$ then for $ k \geq k_o(t)$ $ |c_{n_k}(t)| > f(t), $ and consequently 
$$
d_{n_k}(t) =P_f(c_{n_k})(t) = f(t) \rightarrow f(t). 
$$
If $ |c(t)| = f(t),$ then passing to a convergent subsequence, if necessary, we can assume that $ d_{n_k}(t) \rightarrow f(t)$ or   $ d_{n_k}(t) \rightarrow -f(t).$ Finally, if $ |c(t)| < f(t),$ then 
$ d_{n_k}(t) = P_f(c_{n_k}) = c_{n_k}(t)$ for $ k \geq k_o(t)$ and consequently $ d_{n_k}(t) \rightarrow c(t).$ The proof is complete. 
\end{proof}
\begin{thm}
\label{F1}
Let $(T, \Sigma, \mu)$ be a measure space such that $\mu$ is $ \sigma$-finite. Let $ T = \bigcup_{n=1}^{\infty}T_n, $ where $ T_n \in \Sigma,$ $ T_n \subset T_{n+1}$ and $ \mu(T_n) < \infty.$ Put 
$$ 
L_o(T) = \{ f:T \rightarrow \mathbb{R} : f \hbox{ is } \mu-measurable \}.
$$
Let $ (X , \| \cdot \|)$ be a Fr\'echet space such that $ X \subset L_o(T).$ Assume that for any $\{ x_n\} \subset X$ and $x \in X$ if $ \| x_n-x\| \rightarrow 0,$ 
then $x_n \rightarrow x$ locally in measure. If $ C \subset X$ is $\mu$-sequentially compact then $cl(C)$ is also $\mu$-sequentially compact.
\end{thm}
\begin{proof}
Fix $ \{d_n\} \subset  cl(C).$ Choose for any $n \in \mathbb{N}$ $ c_n \in C$ such that $ \| c_n -d_n \| < \frac{1}{n}.$ By our assumption $ d_n - c_n \rightarrow 0$ locally in measure.
Hence for any $ i \in \mathbb{N} $ there exists a subsequence $ \{n_k\}$ (depending on $i$) such that $ c_{n_k} - d_{n_k} \rightarrow 0 $ $ \mu$-a.e. on $ T_i.$ Applying the  diagonal argument we can assume that $ c_{n_k} - d_{n_k} \rightarrow 0$ $ \mu$-a.e. on $T.$ Since $C$ is $\mu$-sequentially compact, passing to a convergent subsequence if necessary, we can assume that 
$ c_{n_k}(t) \rightarrow c(t) \in \mathbb{R} \cup \{ -\infty, +\infty \}$ $ \mu$-a.e.. Consequently $ d_{n_k}(t) \rightarrow c(t)$ $ \mu$-a.e.,
as required.
\end{proof}
Now we present the main result concerning proximanality of $ \mu$-sequentially compact sets.
\begin{thm}
 \label{proxym}
Let $(T, \Sigma, \mu)$ be a measure space.
Let $ X \subset L_o(T)$ be a Fr\'echet function space equipped with an F-norm $ \| \cdot \|$ satisfying the Fatou property. Let $C \subset X$ be a $\mu$-sequentially compact closed subset of $ X.$ Assume that for any $ \{ c_n\} \subset C,$ if there exists $ M > 0$ such that for any $n\in\mathbb{N}$, $ \| c_n\| \leq M$ and $ c_n(t) \rightarrow c(t) \in \mathbb{R}$ $\mu$-a.e., then $ c \in C.$
Let $ f \in X \setminus C.$ Assume that there exists $ \{ c_n \} \subset C$ such that $ \| f -c_n\| \rightarrow dist(f,C)$ and 
\begin{equation}
\label{restriction} 
\mu(\{ t \in T: \liminf_{n\rightarrow\infty}|c_n(t)|=+\infty \}) =0. 
\end{equation}
Then 
$$
P_C(f) = \{ c \in C: \| f-c\| = dist(f,C)\} \neq \emptyset.
$$
\end{thm}
\begin{proof}
Fix $ f \in X \setminus C$ and $ \{ c_n\} \subset C$ such that $ \| f -c_n\| \rightarrow dist(f,C)$ satisfying  (\ref{restriction}). Since $ C$ is $\mu$-sequentially compact, by our assumptions, there exists a subsequence $\{n_k\}$  and $ c \in C$ such that $ c_{n_k}(t) \rightarrow c(t)$ $\mu$-a.e.. 
Hence $ |(f-c_{n_k})(t)| \rightarrow |(f-c)(t)|$ $\mu$-a.e.. Put for $ k \in \mathbb{N}$ $ g_k(t) = \inf \{ |(f-c_{n_l})(t)|, l \geq k\}.$ Notice that $g_k \uparrow |f-c|$ $\mu$-a.e.
and $ \sup\{ \|g_k\|: k \in \mathbb{N} \} < \infty. $ By the Fatou property, 
$$ 
\| f-c \| = \lim_k \|g_k\|.
$$  
Since $ X$ is a Fr\'echet function space, for each $ k$ $ \| g_k \| \leq \| f-c_{n_k}\|,$ which shows that $ \| f-c\| = dist(f,C).$ Hence $P_C(f) \neq \emptyset,$ as required.
\end{proof}
Now we present some applications of Theorem \ref{proxym}.
\begin{thm}
Let $(T, \Sigma, \mu)$ be a $\sigma$-finite measure space.
Let $ X \subset L_o(T)$ be a Fr\'echet function space equipped with an F-norm $ \| \cdot \|$ satisfying the Fatou property. Let $ C \subset X$ be such that for any $ M >0$, $C \cap B(0,M)= C\cap\{ x \in X: \|x\| \leq M \}$ is compact  in the topology of local convergence in measure. Then $C$ is proximinal in $X.$
\end{thm}
\begin{proof}
Fix $ f \in X \setminus C$ and $ \{ c_n \} \subset C$ such that $ \| f - c_n \| \rightarrow dist(f,C).$ Let $T = \bigcup_{n=1}^{\infty} T_n$ such that $ \mu(T_n) < \infty .$ By the triangle inequality, 
$$  
\|c_n \| \leq dist(f,C) + \|f\| + K\quad\textnormal{with }K>0.
$$
By our assumption for any $i \in \mathbb{N}$ there exists a subsequence $ \{ n_k \} $ and $ c^i \in C|_{T_i}$ such that $ c_{n_k} \rightarrow c^i$ in measure on $ T_i.$ Passing to a convergent subsequence, if necessary, we can asume that  $ c_{n_k} \rightarrow c^i$ in $ \mu$-a.e. on $ T_i$. Applying diagonal argument, we can choose 
a subsequence $ \{ n_k \} $ such that $ c_{n_k} \rightarrow c\in C$ $\mu$-a.e. on $T.$ By Theorem \ref{proxym} $ P_C(f) \neq \emptyset,$ as required. 
\end{proof} 
First we show that (\ref{restriction}) is not very restrictive.
\begin{lem}
 \label{proxym1}
Let $(T, \Sigma, \mu)$ be a $\sigma$-finite measure space.
Let $ X \subset L_o(T)$ be a Fr\'echet function space equipped with an F-norm $ \| \cdot \|$ satisfying the Fatou property.
Assume that for any measurable $ A \subset T$ if $ \mu(A) >0$ then 
$$
\lim_{k\rightarrow \infty } \| k\chi(A)\| = + \infty,
$$ 
where $ \chi(A)$ denotes the characteristic function of $A.$
If $ \{c_n\} \subset X$ is such that $ \| c_n\| < M $ for some $ M \in \mathbb{R}$ and any $n\in\mathbb{N}$ then 
$ \{c_n\} $ holds (\ref{restriction}).
\end{lem}
\begin{proof}
Fix $ \{c_n\} \subset X$ is such that $ \| c_n\| < M $ for some $ M \in \mathbb{R}$ and any $n\in\mathbb{N}.$
Let $A= \{ t \in T: \liminf_{n\rightarrow\infty}|c_n(t)|=+\infty \}.$ Assume to the contrary that $ \mu(A) >0.$ Without loss of generality we can assume that $ \mu(A) < \infty $ and $c_{n}(t) \rightarrow + \infty $ for any $ t \in A.$ Put for $ n \in \mathbb{N},$  $ g_n(t) = \inf_{m\geq n} c_m(t).$ It is clear that $ g_n(t) \uparrow +\infty $ for any
$t \in A.$ Fix $ k \in \mathbb{N}.$ Define for any $ n\in\mathbb{N},$
$$
A_{k,n} = \{ t \in A: g_n(t) > k\}.
$$ 
Observe that $ A_{n,k} \subset A_{n+1,k} $ and $ \bigcup_{n=1}^{\infty} A_{n,k} = A.$ Hence $ \lim_n \mu(A_{n,k}) = \mu(A).$ Notice that 
$$
\| c_n\| \geq \| g_n \chi(A_{n,k}) \| \geq \| k   \chi(A_{n,k}) \|.
$$ 
By the Fatou property 
$$
\lim_n \| k   \chi(A_{n,k}) \| = \|k  \chi(A)\|
$$
and consequently 
$$
\liminf_n \| c_n \| \geq \|k  \chi(A)\|
$$
for any $ k \in\mathbb{N} .$ Since $k$ was arbitrary, $\liminf_n \| c_n \| = +\infty,$ which leads to a contradiction.
\end{proof}
\begin{rem}
A Fr\'echet space $(X, \| \cdot \|)$ is called $ s$-convex for some $ s \in (0,1]$ if  $ \|tx\| = |t|^s\|x\|$ for any $ x \in X$ and $ t \in \mathbb{R}.$ It is clear that any $s$-convex Fr\'echet function space
(in particular any Banach function space) satisfies the assumptions of Lemma \ref{proxym1}.  Also any Orlicz space generated by an increasing function $ \phi$ satisfying $ \lim_{t \rightarrow +\infty} \phi(t) = + \infty$ 
satisfies the assumptions of Lemma \ref{proxym1}.
\end{rem}
Now, applying Theorem \ref{proxym}, Theorem \ref{increasing} and Theorem \ref{convex}, we can show
\begin {thm}
\label{increasing1}
Let $ T = \mathbb{R},$ $ T = \mathbb{R}_+$ or $ T = (a,b).$ 
Let $(T, \Sigma, \mu)$ be a $\sigma$-finite, nonatomic measure space.
Let $ X \subset L_o(T)$ be a Fr\'echet function space equipped with an F-norm $ \| \cdot \|$ satisfying the Fatou property.
Assume that for any measurable $ A \subset T$ if $ \mu(A) >0$ then 
$$
\lim_{k\rightarrow \infty } \| k\chi(A)\| = + \infty,
$$ 
where $ \chi(A)$ denotes the characteristic function of $A.$
LeT $ C $ denote the set of all increasing (decreasing, resp.) functions on $T,$ or let $ C $ denote the set of all convex (concave, resp.) functions on $T.$ If $ C \subset X$ then $C$ is proximinal in $X$
(in particular $C$ is closed in $X).$ 
\end{thm}
\begin{proof}
Assume that $C$ is the set of all increasing functions on $T.$ Fix $ f \in X$ and $ \{ c_n \} \subset C$ be such that $ \| f - c_n\| \rightarrow 0.$ By Theorem \ref{increasing}, there exists  a subsequence 
$ \{ n_k\}$ such that $ c_{n_k}(t) \rightarrow c(t) \in [-\infty, +\infty]$  $ \mu$-a.e.. Since the sequence $ \| c_n\|$ is bounded, by our assumptions $c(t) \in \mathbb{R}$ $\mu$-a.e..
It is easy to see that $c$ is also an increasing function. By Theorem \ref{proxym} we get our result. If $C$ is the set of all decreasing functions on $T$ applying the above reasoning to $-C$ we get our claim.
\newline
If $C$ is the set of all convex (concave, resp.) functions on $T$ applying Theorem \ref{convex} instead of Theorem \ref{increasing} and reasoning as above, we get our conclusion.
\end{proof}
\begin{rem}
\label{var1}
Applying Theorem \ref{increasing1} and Theorem \ref{constructions} we can find other examples of proximinal subsets of Fr\'echet function spaces $X$ contained in $ L_o(T).$ 
\end{rem}
\begin{thm}
\label{pieces}
Let $ T = (a,b)$ and let $ \{t_n\} \subset (a,b),$ $t_n < t_{n+1}$ be a partition of $(a,b).$  Let for $ n \in \mathbb{N},$ $(T_n=(t_n,t_{n+1}), \Sigma_n, \mu_n)$ be a $\sigma$-finite, nonatomic measure space.
Let $$ 
L_{o,n} = \{f: T_n \rightarrow \mathbb{R}, f \hbox{ is }\mu_n\hbox{-measurable}\}.
$$
Let $ X_n \subset L_{o,n}$ be a Fr\'echet space equipped with an F-norm $ \| \cdot \|_n.$ 
Define
$$
X = \{ f:(a,b)\rightarrow \mathbb{R}: f|_{T_n} \in X_n\}
$$
and for $ f \in X$
$$
\| f\| = \sum_{n=1}^{\infty} \frac{\|f|_{T_n}\|_n}{2^n(1+\|f|_{T_n}\|_n)}.
$$
Let for $ n \in \mathbb{N}$ $ C_n \subset X_n$ be a proximinal set in $ X_n.$ Put 
$$
C = \{ f\in{X}: f|_{T_n} \in C_n\}.
$$
Then $ C$ is proximinal in $X.$
\end{thm}
\begin{proof}
First observe that for any $ f \in X$ and $ c \in C,$
$$
\| f - c\| = \sum_{n=1}^{\infty} \frac{\|(f-c)|_{T_n}\|_n}{2^n(1+\|(f-c)|_{T_n}\|_n)} 
$$
$$
\geq \sum_{n=1}^{\infty} \frac{dist_n(f|_{T_n},C_n)}{2^n(1+dist_n(f|_{T_n},C_n)))},
$$
where for any $g \in X_n$ $dist_n(g,C_n)$ denotes the distance of $g$ to $C_n$ with respect to $ \| \cdot \|_n.$ 
For $ n \in \mathbb{N}$ fix $ c_n \in P_{C_n}(f|_{T_n}).$ Define $ c \in C$ by $c(t) = c_n(t) $ for $ t \in T_n$. Since $ T_n  \cap T_m = \emptyset $ for $ n\neq m,$ it is easy to see $c$ is well defined. Notice that,  
$$
\| f - c\| =\sum_{n=1}^{\infty} \frac{dist_n(f|_{T_n},C_n)}{2^n(1+dist_n(f|_{T_n},C_n)))},
$$
which shows that $ c \in P_C(f).$ The proof is complete.
\end{proof}
\begin{rem}
\label{var2}
If for any $ n\in \mathbb{N}$
$ C_n $ denote the set of all increasing (decreasing, resp.) functions on $T_n,$ or $ C_n $ denote the set of all convex (concave, resp.) functions on $T_n$ and $X_n \subset L_{o,n}$ is a Fr\'echet function space satysfying the assumptions of Theorem \ref{increasing1} (see also Rem. \ref{var1}), then by Theorem \ref{increasing} we can apply Theorem \ref{pieces} to $ \{ X_n\} $ and $ \{C_n\}.$
\end{rem}
Now we apply Theorem \ref{pieces} and Remark \ref{var2} to construct Chebyshev subsets in Fr\'echet function spaces.
\begin{thm}
\label{pieces1}
Let $X_n,$ $C_n,$ $\| \cdot \|_n$ for any $n\in\mathbb{N}$ and $C$, $X$ be as in Theorem \ref{pieces}. Assume that for any $ n \in \mathbb{N}$ the set $C_n$ is a Chebyshev subset of $X_n$ with respect to $ \| \cdot \|_n$. Then $C$ is a Chebyshev subset of $X$ with respect to $ \| \cdot \|.$ 
\end{thm}
\begin{proof}
By Theorem \ref{pieces} for any $ f \in X$ $ P_C(f) \neq \emptyset . $ Moreover for any $ c \in C,$
$$
\| f - c\| = \sum_{n=1}^{\infty} \frac{dist_n(f|_{T_n},C_n)}{2^n(1+dist_n(f|_{T_n},C_n)))}.
$$
This implies that for any $ n \in \mathbb{N}$ $ c|_{T_n} \in P_{C_n}(f|_{T_n}).$ Since $C_n$ is a Chebyshev set in $ X_n,$ $ c|_{T_n}$ is uniquely determined for any $ n \in \mathbb{N}.$ Hence $ c \in P_C(f)$ is uniquely determined, which shows that $C$ is a Chebyshev subset of $X.$
\end{proof}
Now we present an example of a Chebyshev set in a Fr\'echet function space.
\begin{exam}
\label{interesting}
Let $ T = (a,b)$ and let $ \{t_n\} \subset (a,b),$ $t_n < t_{n+1}$ be a partition of $(a,b).$ 
Let for $ n \in \mathbb{N},$ $(T_n=(t_n,t_{n+1}), \Sigma_n, \mu_n)$ be a $\sigma$-finite, nonatomic measure space.
Let $$ 
L_{o,n} = \{f: T_n \rightarrow \mathbb{R}, f \hbox{ is }\mu_n\hbox{-measurable}\}.
$$
Let $ X_n \subset L_{o,n}$ be a strictly convex, reflexive Banach space equipped with a norm $ \| \cdot \|_n.$ Fore example $ X_n = L^{p_n}(T_n),$ where $ 1 < p_n < \infty.)$ 
Define
$$
X = \{ f:(a,b) \rightarrow \mathbb{R}: f|_{T_n} \in X_n\}
$$
and for $ f \in X$
$$
\| f\| = \sum_{n=1}^{\infty} \frac{\|f|_{T_n}\|_n}{2^n(1+\|f|_{T_n}\|_n)}.
$$
Let for $ n \in \mathbb{N},$ $ C_n \subset X_n$ be a closed, convex subset of $ X_n.$ Put 
$$
C = \{ f:(a,b) \rightarrow \mathbb{R}: f|_{T_n} \in C_n\}.
$$
Since for any $n\in\mathbb{N},$ $ X_n$ is reflexive and $C_n$ is closed and convex it follows that $C_n$ is a Chebyshev  subset of $X_n$ with respect to $\|\cdot\|_n$. Hence $C_n$ is a Chebyshev set with respect to $F$-norm $\frac{\|\cdot\|_n}{1+\|\cdot\|_n}$. By Theorem \ref{pieces1} $C$ is a Chebyshev subset of $X.$  
In particular, $ C_n $ can be the set of all increasing (decreasing, resp.) functions on $T_n$ or $ C_n $ can be the set of all convex (concave, resp.) functions on $T_n.$ It is clear that $C$ is a Chebyshev set.
\end{exam}

Now we research $\mu$-compactness in Banach function spaces. 

\begin{defn}
	A nonempty subset $C\subset{Z}$ is said to be $\mu$-compact if and only if for any sequence $(c_n)\subset{C}$ there exist $(n_k)\subset\mathbb{N}$ and $c\in{C}$ such that $c_{n_k}$ converges to $c$ in
 measure $\mu.$
\end{defn}

\begin{rem}
	Let $C$ be a nonempty subset of $L^*=\{x^*:x\in{L^0}\}$, where $x^*$ is a decreasing rearrangement of $x\in{L^0}$. Then, by Theorem \ref{increasing} and Theorem \ref{constructions} we get immediately $C$
 is $\mu$-sequentially compact. Additionally, if we assume that $c(\infty)=0$ for any $c\in{C}$, then $C$ is $\mu$-compact. Indeed, since $C$ is $\mu$-sequentially compact for any sequence $(c_n)\subset{C}$ there
 exist $(n_k)\subset\mathbb{N}$ and $c\in{C}$ such that $c_{n_k}$ converges to $c$ $m$-a.e. Hence, $c_{n_k}$ converges to $c$ locally in measure and by assumption that each element $a\in C$, $a=a^*$ and
 $a(\infty)=0$ it follows that $c_{n_k}$ converges to $c$ in measure.
\end{rem}

Now we present properties of $\mu$-compact sets. The proof of the below result follows immediately by definition of $\mu$-compactness and similarly as in case of Theorem \ref{constructions}.
\begin{thm}
	\label{constructions2}
	Let $(T, \Sigma, \mu)$ be a measure space. 
	\begin{itemize}
		\item[$(a)$] If $C$ is $\mu$-compact then any nonempty set $ D \subset C$ is $\mu$-compact.
		\item[$(b)$] If $ C_i$  for $ i=1,..n$ are $\mu$-compact, then $C = \bigcup_{j=1}^{n}C_i$ is $\mu$-compact.
		\item[$(c)$] Let $C_1$ and $C_2$ be $\mu$-compact subsets such that $C_i\subset\Pi_{t\in T} Z_t$ for $i=1,2$, where  $Z_t = [f(t), +\infty)$ and $f(t) >0$ for $t \in T.$ Then $ aC_1 + bC_2$ with $ a, b\geq 0$ are  $\mu$-compact sets. Additionally, if $\mu(T)<\infty$ then $C_1C_2 = \{ c_1c_2 : c_1 \in C_1, c_2 \in C_2 \}$ is $\mu$-compact. 
		\item[$(d)$] If $C_1$ and $C_2$ are $\mu$-compact, then $ C_1\vee{C_2} = \{c_1\vee{}c_2: c_1 \in C_1, c_2 \in C_2 \}$ and $ C_1\wedge{C_2} = \{ c_1\wedge{}c_2: c_1 \in C_1, c_2 \in C_2 \}$ are $\mu$-compact sets, where $(c_1\vee{}c_2)(t) = \max(c_1(t),c_2(t))$ and $(c_1\wedge{}c_2)(t) = \min(c_1(t),c_2(t)).$
		\item[$(e)$] Let $ f : T \rightarrow \mathbb{R}_+$ be a function. If $C$ is $\mu$-compact, then $ P_f(C)$ is also $\mu$-compact, where for $ c \in C$ and $t \in T,$ $P_f(c)(t) = sgn(c(t))\{|c(t)|\wedge{f(t)}\}$.
	\end{itemize}
\end{thm}

\begin{thm}\label{thm:app-com}
	Let $E$ be a symmetric Banach function space with semi-Fatou property and $H_g$ property and let $C\subset{E}$ be nonempty $\mu$-compact subset of $E$. Then for every $x\in{E}\setminus{C}$ and for any minimizing sequence $(c_n)\subset{C}$ of $\dist(x,C)$ there exist a subsequence $(c_{n_k})\subset(c_n)$ and $c\in{P}_{C}(x)$ such that $$\norm{c_{n_k}-c}{}{}\rightarrow{0}.$$
\end{thm}

\begin{proof}
	Assume that $(c_n)\subset{C}$ is a minimizing sequence of $\dist(x,C)>0$, i.e. 
	\begin{equation}\label{equ:thm:H_g:in:approx}
	\lim_{n\rightarrow\infty}\norm{x-c_n}{}{}=\inf_{d\in{C}}\norm{x-d}{}{}>0.
	\end{equation}
	Since $C$ is $\mu$-compact there exist $(c_{n_k})\subset(c_n)$ and $c\in{C}$ such that $c_{n_k}$ converges to $c$ globally in measure. Therefore, $x-c_{n_k}$ converges to $x-c$ globally in measure. Consequently, by assumption that $E$ has $H_g$ property and by Lemma 3.8 \cite{CieKolPlu} it follows that 
	\begin{equation*}
	\norm{x-c}{}{}\leq\liminf_{k\rightarrow\infty}\norm{x-c_{n_k}}{}{}.
	\end{equation*}
	Hence, by condition \eqref{equ:thm:H_g:in:approx} we obtain
	\begin{align*}
	\dist(x,C)\leq\norm{x-c}{}{}\leq\liminf_{k\rightarrow\infty}\norm{x-c_{n_k}}{}{}\leq\lim_{k\rightarrow\infty}\norm{x-c_{n_k}}{}{}=\dist(x,C).
	\end{align*}
	Finally, according to assumption that $E$ has $H_g$ property we get 
	\begin{equation*}
	\norm{c_{n_k}-c}{}{}\rightarrow{0}.
	\end{equation*}
\end{proof}

\begin{coro}
	Let $\phi$ be a concave increasing function such that $\phi(0^+)=0$ and let $C\subset{\Lambda_\phi}$ be a nonempty $\mu$-compact set and $x\in\Lambda_\phi\setminus{C}$. Then for any minimizing sequence $(c_n)\subset{C}$ of $\dist(C,x)$ there exist $(n_k)\subset\mathbb{N}$ and $c\in{P_C(x)}$ such that
	\begin{equation*}
	\norm{c_{n_k}-c}{\Lambda_{\phi}}{}\rightarrow{0}.
	\end{equation*}
\end{coro}

\begin{proof}
	By Corollary 1.3 \cite{ChDSS} it follows that $\lambda_\phi$ has the Kadec Klee property for global convergence in measure. Since $\Lambda_{\phi}$ satisfies Fatou property \cite{Loren} we have $\Lambda_\phi$ has also semi-Fatou property and by Theorem \ref{thm:app-com} we complete the proof.
\end{proof}

\begin{coro}
	Let $0<p<\infty$ and $w\geq{0}$ be weight function on $I$. Let $C\subset{\Lambda_\phi}$ be a nonempty $\mu$-compact set and $x\in\Gamma_{p,w}\setminus{C}$. Then for any minimizing sequence $(c_n)\subset{C}$ of $\dist(C,x)$ there exist $(n_k)\subset\mathbb{N}$ and $c\in{P_C(x)}$ such that
	\begin{equation*}
	\norm{c_{n_k}-c}{\Gamma_{p,w}}{}\rightarrow{0}.
	\end{equation*}
\end{coro}

\begin{proof}
	Immediately, by Theorem 4.1 \cite{CieKolPlu} we obtain that $\Gamma_{p,w}$ has the Kadec Klee property for global convergence in measure. Furthermore, it is well known that $\Gamma_{p,w}$ has Fatou propert and so it has also semi-Fatou property. Hence, by Theorem \ref{thm:app-com} it follows that for any minimizing sequence $(c_n)\subset{C}$ there exist $(n_k)\subset\mathbb{N}$ and $c\in{P_C(x)}$ such that $\norm{c_{n_k}-c}{\Gamma_{p,w}}{}\rightarrow{0}$.
\end{proof}

\section{Property (S) and continuity of metric projection operator} 

First we introduce the property $(S).$

\begin{def}
\label{Kadec}
Let $(T, \Sigma, \mu)$ be a measure space.
Let $ X \subset L_o(T)$ be a Fr\'echet function space equipped with an F-norm $ \| \cdot \|.$
It is said that $X$ has property {\bf (S)} if there exists $ 0<p < \infty $ such that for any $ \{ c_n \} \subset X,$ $ c \in X$ and $f \in X,$
$$
\| f-c_n\|^p = \|f-c\|^p + \|c_n-c\|^p + o(1)
$$
provided $ c_n \rightarrow c$ locally in measure.
\end{def}

Now we present an some Banach function spaces with property $(S)$.
\begin{exam}
Let $0<\lambda<\infty$ and $\phi$ be concave increasing function on $[0,\alpha)$ with $\phi(0^+)=0$. Replacing $x$ by $f-c$ and $y_n$ by $c_n-c$ in Proposition 1.2 \cite{KPS} we get the Lorentz space $\Lambda_\phi$ with the property $(S)$ equipped with the norm given by
\begin{equation*}
\norm{x}{\Lambda_\phi}{}=\int_{0}^{\alpha}{x^*(t)\phi'(t)}dt.
\end{equation*}
\end{exam}
\begin{exam}
Let $p\geq{}1$ and $w$ be a nonegative weight function on $[0,\alpha)$ with $0<\alpha<\infty$. Taking $x=f-c$ and $y_n=c_n-c$ in the proof of Theorem 4.1 \cite{CieKolPlu} we can easily observe that the Lorentz space $\Gamma_{p,w}$ holds the property $(S)$.
\end{exam}

The next examples show some special Fr\'echet function spaces which satisfy the property $(S)$.
\begin{exam}
\label{one}
Let $(T, \Sigma, \mu)$ be a finite measure space. For $ f \in L_o(T) $ define 
$$
\| f\| = \int_T \frac{|f(t)|}{1+|f(t)|} d\mu(t).
$$
It is clear that for any $ \{ c_n \} \subset X$ and $ c \in X$, $ \| c_n - c\| \rightarrow 0$ if and only if $c_n \rightarrow c$ locally in measure.
Hence in this case property (S) is satisfied.
\end{exam}
\begin{exam}
\label{two}
Let $(T, \Sigma, \mu)$ be a $\sigma$-finite measure space. Assume that $ T = \bigcup_{n=1}^\infty T_n, $ where $ T_n \subset T_{n+1}$ and $ \mu(T_n) < \infty. $
Define for $f \in L_o(T)$ 
$$
\|f\| = \sum_{n=1}^{\infty} \frac{1}{2^n} \int_{T_n} \frac{|f(t)|}{1+|f(t)|} d\mu(t).
$$
It is clear that for any $ \{ c_n \} \subset X$ and $ c \in X$, $ \| c_n - c\| \rightarrow 0$ if and only if $c_n \rightarrow c$ locally in measure.
Hence in this case property (S) is satisfied.
\end{exam}

The next theorem shows a large class of Fr\'echet function spaces satisfying the Kadec-Klee property for local convergence in measure.
\begin{thm}
\label{FS1}
Let for $ n \in \mathbb{N}$ $(T_n, \Sigma_n, \mu_n)$ be a sequence of measure spaces such that $ \mu(T_n) < \infty $ and $ T_n \cap T_m = \emptyset $ if $ m \neq n.$ Let  $ T = \bigcup_{n=1}^{\infty} T_n$ be equipped with the measure
$$
\mu(S) = \sum_{n=1}^{\infty} \mu_n(S\cap T_n).
$$
Let 
$$ 
L_{o,n} = \{f: T_n \rightarrow \mathbb{R}, f \hbox{ is }\mu_n\hbox{-measurable}\}.
$$
Let $ X_n \subset L_{o,n}$ be a Fr\'echet function space equipped with an F-norm $ \| \cdot \|_n$ with the Fatou property. 
Define
$$
X = \bigoplus_{n=1}^{\infty} X_n \subset L_o(T)
$$
and for $ f= (f_1,f_2,...)  \in X$
$$
\| f\| = \sum_{n=1}^{\infty} \frac{\|f_n\|_n}{2^n(1+\|f_n \|_n)}.
$$
Assume that for any $n \in \mathbb{N},$ $ \| \cdot \|_n$ satisfies property $(S)$ with $p_n \in (1, +\infty).$ Then $(X, \| \cdot \|)$ satisfies the Kadec-Klee property for local convergence in measure.
\end{thm}
\begin{proof}
Fix $ c =(c_1,c_2, ...) \in X$ $f= (f_1,f_2,...)  \in X$ and $ \{ c_n\} \subset X$ such that $ c_n \rightarrow c$ locally in measure. Assume to the contrary that $ \| f-c_n \| \rightarrow \| f-c\| $ and $ \| c_n -c\|\nrightarrow{0}$. Passing to a subsequence and relabelling if necessary, we can assume
\begin{equation}
\label{d} 
\| c_n -c\| > d>0
\end{equation}
for any $ n \in \mathbb{N}.$ Since for any $j \in \mathbb{N},$ $ \mu(T_j) < \infty,$ again passing to a convergent subsequence, if necessary, and applying diagonal argument we can assume that $ c_n(t) \rightarrow c(t)$ $\mu$-a.e. Hence, $ |f(t) -c_n(t)| \rightarrow |f(t)-c(t)|$  $\mu$-a.e. By the Fatou property, for any $j \in \mathbb{N}$ we get
$$ 
\| f_j - (c)_j\|_j \leq \liminf_n \| f_j - (c_n)_j\|_j.
$$
Since $ \| f-c_n\| \rightarrow \|f-c\|,$ by the above inequality, for any  $j \in \mathbb{N},$ 
$$
\lim_n \| f_j - (c_n)_j\|_j =\| f_j - c_j\|_j.
$$
By property (S) applied to $ \| \cdot \|_j,$ $ (c_n)_j , c_j$ and $ f_j$ we get 
$$ 
\| (c_n)_j - c_j\|_j \rightarrow_n 0.
$$
Consequently $ \| c_n - c\| \rightarrow_n 0,$ which is a contradiction with (\ref{d}).
\end{proof}
Now we apply the Kadec-Klee property for local convergence in measure to the problem of continuity of the metric projection operator.
\begin{def}
\label{projop}
Let $ X$ be a Fr\'echet space and let $ C\subset X$ be a proximinal subset of $X.$ Let for $ x \in X, $ 
$$
P_C(x) = \{ c \in C: \| x-c\| = dist(x,C)\}.
$$
The metric projection operator is a mapping $P$ from $X$ to $ 2^C$ defined by 
$ Pf = P_C(f).$ 
\end{def}
\begin{thm}
\label{cont1}
Let $\mu$ be a $\sigma$-finite measure and let $(X,\| \cdot \|) \subset L_o(T) = L_o(T, \Sigma, \mu)$ be a Fr\'echet function space with the Fatou property and $X\in(H_l).$ Let $C \subset X$ be a proximinal subset of $X.$ Assume for any $g \in X$ and $ c_n \subset C$ if $ \| g-c_n\| \rightarrow  dist(g,C)$ then there exists $c \in C$ and a subsequence $ \{ c_{n_k}\}$ such that $ c_{n_k} \rightarrow c$ locally in measure. Let $ \{ f_n\} \subset X$, $f \in X$ and $ \| f_n-f\| \rightarrow 0.$  Then
$$
\lim_n (\sup\{ dist(c, P_C(f)): c \in P_C(f_n)\}) =0.
$$
\end{thm}
\begin{proof}
Assume to the contrary that there exist $ \{ f_n\} \subset X$ and $f \in X$ such that $ \| f_n-f\| \rightarrow 0$ and 
$$
\limsup_n (\sup\{ dist(c, P_C(f)): c \in P_C(f_n)\}) >0.
$$
Passing to a subsequence if necessary, we can assume that 
$$
\sup\{ dist(c, P_C(f)): c \in P_C(f_n)\} > d >0
$$ 
for any $n \in \mathbb{N}.$ For any $ n \in \mathbb{N}$ select $ c_n \in P_C(f_n)$ such that
$$
dist(c_n, P_C(f)) > \frac{d}{2}.
$$  
Notice that
$$
dist(f,C) \leq \| f-c_n\| \leq (dist(f_n,C) + \|f_n-f\|)  \rightarrow_n dist(f,C).
$$  
Hence, since $ \| f_n-f\| \rightarrow 0,$ $ \| f-c_n\| \rightarrow dist(f,C).$ By our assumptions there exists a subsequence $ \{ n_k\}$ and $ c \in C,$ such that $ c_{n_k} \rightarrow c$ locally in measure. Since $ \mu$ is $\sigma$-finite, applying the diagonal argument and passing to a convergent subsequence, if necessary we can assume that $ c_{n_k}(t) \rightarrow c(t)$ $\mu$-a.e.. By the Fatou property, 
\begin{align*}
\|f-c\| &\leq \liminf_{k\rightarrow\infty} \| f -c_{n_k}\| \leq \liminf_{k\rightarrow\infty} \| f-f_{n_k}\| + \lim_{k\rightarrow\infty} dist(f_{n_k},C)\\
&= dist(f,C) \leq \|f-c\|.
\end{align*} 
Consequently, $ c \in P_C(f)$ and $ \|f-c_{n_k} \| \rightarrow \| f-c\|.$  
By assumption that $X\in(H_l)$, it follows that $ \| c_{n_k} -c\| \rightarrow 0.$ Hence,
$$
\frac{d}{2} < dist(c_{n_k}, P_C(f)) \leq \| c_{n_k} - c\| \rightarrow 0,
$$
which is a contradiction.
\end{proof}
\begin{thm}
\label{cont2}
Let $X$, $ \{f_n\}\subset X,$ $f \in X$ and $C \subset X$ be such as in Theorem \ref{cont1}. 
Assume that $ P_C(f)$ is a singleton. Then 
$$
\lim_n d_H (P_C(f), P_C(f_n)) =0,
$$
where $d_H$ denote the Hausdorff distance between $ P_C(f)$ and $P_{C}(f_n).$  
\end{thm}
\begin{proof}
By Theorem \ref{cont1}
$$
\lim_n (\sup\{ dist(d, P_C(f)): d \in P_C(f_n)\}) =0.
$$
We need to show that 
$$
\lim_n (\sup\{ dist(c, P_{C}(f_n)): c \in P_C(f)\}) =0.
$$
By our assumption $ P_C(f) = \{ c \}.$ Hence, the above inequality reduces to 
 $$
\lim_{n\rightarrow\infty} dist(c, P_{C}(f_n))=0.
$$
Assume that this is not true. Reasoning as in Theorem \ref{cont1} we can choose a subsequence $\{ c_{n_k} \} \subset C$ such that $ c_{n_k} \in P_C(f_{n_k})$ and $d>0$ such that for any $k\in\mathbb{N}$,
$$
\| c_{n_k} - c\| > \frac{d}{2}. 
$$
Notice that
$$
dist(f,C) \leq \|f-c_{n_k}\| \leq  dist(f_{n_k},C) + \| f- f_{n_k}\| .
$$
Hence,
$$
\| f-c_{n_k}\| \rightarrow dist(f,C).
$$
Reasoning as in Theorem \ref{cont1} passing to a convergent subsequence if necessary we can assume that $ \| c_{n_k} - u\| \rightarrow 0$ for some $ u \in C.$ 
It is clear that $ u \in P_C(f).$ Since $  P_C(f) = \{ c \},$ $u=c,$ which leads to a contradiction.
\end{proof}
\begin{exam}
Let $X = \mathbb{R}$ and let $ C= \{ -1,1\}.$ Observe that $ P_C(0) = C$ and $ P_C(\frac{1}{n}) = \{1 \}.$ 
Hence 
$$
\sup\{ dist(c, P_{C}(\frac{1}{n})): c \in P_C(0)\} = |-1-1| =2,
$$
which shows that the assumption that $ P_C(f)$ is a singleton in Theorem \ref{cont2} is essential.
\end{exam}
As an immediate consequence of Theorem \ref{cont2} we get the following result.
\begin{thm}
\label{cont3}
Let $X$, $\{f_n\}\subset X,$ $f \in X$ and $C \subset X$ be such as in Theorem \ref{cont1}. 
Assume that for any $g \in X$ $ P_C(g)$ is a singleton. Then the mapping $ g \rightarrow P_C(g) \in C$ is continuous. 
\end{thm}

Now we present in view of  Theorem \ref{cont1} the full criteria of the continuity of the metric projection in Fr\'echet function space with the Kadec-Klee property for local convergence in measure with respect to the Hausdorff distance.

\begin{thm}
\label{cont4}
Let $X$ , $ \{f_n\}\subset X,$ $f \in X$ and $C \subset X$ be such as in Theorem \ref{cont1}. 
Then  
$$
\lim_n d_H (P_C(f), P_C(f_n)) =0,
$$
if and only if for any $\epsilon>0$ there exists $\delta_\epsilon>0$ such that for any $c\in P_C(f)$ if $ \| f-f_n\| <\delta_\epsilon$ then $ P_C(f_n) \cap B(c,\epsilon) \neq \emptyset $, where $ B(c,\epsilon) = \{ x \in X: \|x-c\| \leq\epsilon\}.$
\end{thm}
\begin{proof}
Let $\epsilon>0$. Then there exists $\delta_\epsilon>0$ such that for any $c\in P_C(f)$ we have $dist(c,P_C(f_n))\leq\epsilon$ whenever $ \| f-f_n\| <\delta_\epsilon$. Since $\|f-f_n\|\rightarrow{0}$, there exists $n_0\in\mathbb{N}$ such that for any $n\geq{n_0}$ we get $\|f-f_n\|<\delta_\epsilon$ and consequently $dist(c,P_C(f_n))\leq\epsilon$ for all $c\in{P_C(f)}$. Therefore, for all $n\geq{n_0}$ we get
\begin{equation*}
\sup\{dist(c,P_C(f_n)):c\in P_C(f)\}\leq\epsilon.
\end{equation*}
Thus, by Theorem \ref{cont1} we conclude that
\begin{equation*}
\lim_{n\rightarrow\infty}d_H(P_C(f),P_C(f_n))=0.
\end{equation*}
The converse implication follows immediately by definition of the Hausdorff distance. Indeed, we have
\begin{equation*}
\lim_{n\rightarrow\infty}\sup\{dist(c,P_C(f_n)):c\in P_C(f)\}=0.
\end{equation*}
Now, taking $\epsilon>0$ we may find $n_\epsilon\in\mathbb{N}$ such that for any $n\geq{n_\epsilon}$ and $c\in{P_C(f)}$ we get
\begin{equation*}
dist(c,P_C(f_n))\leq\epsilon.
\end{equation*}
Consequently, by assumption that $\norm{f-f_n}{}{}\rightarrow{0}$ there exists $\delta_\epsilon>0$ such that if $\norm{f-f_n}{}{}<\delta_\epsilon$ then $n\geq{n_\epsilon}$ and for all $c\in{P_C(f)}$ we obtain
\begin{equation*}
B(c,\epsilon)\cap{}P_C(f_n))\neq\emptyset.
\end{equation*}
\end{proof}

\section{Properties of metric projection onto one-dimensional subspaces of sequence Lorentz spaces}

First we recall a well-known result for a sake of completeness and convenience of the reader. 
\begin{thm}
\label{classical}
Let $X$ be a Banach space, $ x \in X$ and let $ Y \subset X$ be a linear subspace. Assume that $dist(x,Y) >0.$ Then $y \in P_Y(x)$ if and only if there exists $ f \in S_{X^*}$ such that $f(x-y) = dist(x,Y)$ and $f(y) =0.$
\end{thm}
If $Y$ is a finite-dimensional subspace of a real Banach space $X$ a stronger version of Theorem \ref{classical} holds.
\begin{thm}
\label{singer}
(see \cite{Singer} , p.170) Let $Y$ be an n-dimensional subspace of a real Banach space $ X$ an let $ x \in X \setminus Y.$ Then $y \in  P_Y(x)$ if and only if there exists $ j \in \{ 1,...,n+1\}, $ $ f_1,...,f_j \in ext(S(X^*)),$
positive numbers $ \lambda_1,...,\lambda_j$ such that for $i\in\{1,...,j\}$,
$$
f_i(x-y) = \| x-y\|
$$
and 
$$
\sum_{i=1}^j (\lambda_i f^i)_Y =0.
$$
\end{thm}
By (\cite{SudWoj}, Th. 5.8) we can immediately deduce
\begin{thm}
\label{sudwoj}
Let $Y$ be an n-dimensional subspace of a real Banach space $ X$ an let $ x \in X \setminus Y$ and $y_o \in P_Y(x).$ If the minimal number $j$ satisfying the requirements of Theorem \ref{singer} is equal to 
$ n+1$ then there there exists $ r > 0$ such that for any $y \in Y$
$$ 
\| x -y \| \geq \|x-y_o\| + r\|y -y_o\|.
$$
By the Freud Theorem (see \cite{Chen}, p. 82) if $x_n \in X$ and $ \|x_n-x\| \rightarrow 0,$ and $ y_n \in P_{Y}(x_n),$ then 
$$
\| y_n -y\| \leq \frac{2\|x_n-x\|}{r}
$$
\end{thm}
In the sequel we also need
\begin{thm}
\label{deutsch}
(see \cite{Deutsch}, Th. 4.5). Let $X$ be a Banach space and $Y \subset X$ a one-dimensional subspace. Then $Y$ does not admit a continuous metric selection if and only if there exists $ x \in X$ such that 
$ 0 \in P_Y(x)$ and disjoint compact intervals $ I_1, I_2$ and two sequences $ \{ x_n\} $ and $ \{ y_n\}$ converging to $x$ such that for any $ n \in \mathbb{N}$ $P_Y(x_n) \subset I_1$ and 
$P_Y(w_n) \subset I_2.$ 
\end{thm}
\begin{lem}
\label{general}
Let $X$ and $Y$ be two Banach spaces and let $X_1 \subset X$ be a closed subspace of $X$ and $Y_1 \subset Y$ be a closed subspace of $Y.$ Assume that $ T: X \rightarrow Y$ is a linear surjective isometry such that 
$ T(X_1) = Y_1.$ Then for any $ x \in X,$ $ P_{Y_1}(Tx) = T(P_{X_1}(x)).$
\end{lem} 
\begin{proof}
Follows immediately by definitions of $P_Y$ and $T.$
\end{proof}
\begin{lem}
\label{Czebyszew}
Let $Y$ be a closed subset of a Banach space $X$ such that $dim(Span(Y))$ is finite. Assume $x\in X$ and $P_Y(x) =\{y\}.$ If $ x_n \in X$ and $ \|x_n-x\| \rightarrow 0,$ then for any $ y_n \in P_Y(x_n),$  we have $\|y_n-y\| \rightarrow 0.$
\end{lem}
\begin{proof}
Assume to the contrary, that
there exist $ \{ x_n\} \subset X,$ $y_n \in P_Y(x_n)$ and $x \in X$ such that $ P_Y(x) =\{y\},$   $ x_n \rightarrow x$ and  $\{y_n\}$ does not converge to $y.$ Passinng to a subsequence, if necessary, we can assume that 
there exists $ d> 0$ such that $ \| y_n -y\| > d.$  Since $ x_n \rightarrow x,$ $\{ y_n\} $ is bounded. Since $dim(Span(Y))<\infty$ and $Y$ is closed, passing to a convergent subsequence, if necessary, we can assume that
$ y_n \rightarrow z \in Y.$ By the continuity of the function $ x \rightarrow dist(x,Y)$ we get $ \| x -z\| = dist(x,Y).$ Since $P_Y(x) = \{y\},$ $ y = z,$ which leads to a contradiction. 
\end{proof}
\begin{rem}
Observe that Lemma \ref{Czebyszew} cannot be generalized to the case of $Y$ being closed subspaces of Banach spaces. In fact, in \cite{Brown1} A. L. Brown constructed a strictly convex reflexive space
having a Chebyshev subspace $Y$ of codimension two such that the metric projection operator  $ x \rightarrow P_Y(x) \in Y$ is not continuous. (Since $ Y$ is a Chebyshev subspace we treat $ P_Y(x)$ as an element from $Y.)$
\end{rem}
\begin{lem}
 \label{Izbior}
Let $ X$ be a Banach space and let $ Y$ be a proper $n$-dimensional subspace of $X.$ Assume that $Y $ is not contained in the intersection of kernels of $n$ linearly independent functionals from the set 
$ext(S(X^*)$, where $ext(S(X^*)$ denotes the set of all extreme point of the unit sphere in $X^*.$ Then for any $x \in X$ there exists $ y_x \in Y$ and $ r_x >0$ depending only on $x$ such that for any $y \in Y,$
$$
\| x - y \| \geq \| x - y_x\| + r\|y-y_x\|.
$$ 
\end{lem}
\begin{proof}
Follows from Theorem \ref {sudwoj}.
\end{proof}
Now, motivated by (\cite{Deutsch}, Th. 6.3 and Cor. 6.6) we restrict ourselves to the case of one-dimensional subspaces of Lorentz sequence spaces.
Let $w$ be a weight function, i.e. $w=(w(1),..w(n),...)$ is a decreasing sequence of positive numbers such that 
\begin{equation} 
\label{sum}
\sum_{n=1}^{\infty} w(n) = + \infty.
\end{equation}
The Lorentz sequence space $d(w,1)$ is the collection of all real sequences $ x=\{ x(n)\},$ 
such that 
$$ 
\| x\|_{w,1} = \sum_{n=1}^{\infty} x^*(n)w(n) < \infty, 
$$
where $x^*$ denotes the decreasing rearrangement of $x.$ Notice that by (\ref{sum}), 
\begin{equation}
\label{c0}
d(w,1) \subset c_o.
\end{equation}
It is well-known that $d(w,1)$ is a Banach space under the norm $ \| \cdot \|_{w,1}.$
\newline
The Marcinkiewicz sequence space $d^*(w,1)$ consists of all real sequences $x = \{ x(n)\}$ such that 
$$
\|x\|_W =\sup_n \frac{ \sum_{j=1}^n x^*(j)}{W(n)} < \infty,
$$ 
where $W(n) = \sum_{j=1}^n w(j).$ 
It is well-known that $d^*(w,1)$ is the dual space of $d(w,1).$ For more details of the Lorentz and Marcinkiewicz spaces the reader is referred to \cite{KamLee,KPS,LinTza}. Now we characterize $1$-dimensional subspaces of $d(w,1)$ under which there exists a continuous metric selection. By Lemma \ref{Czebyszew} and Lemma \ref{Izbior}  if $Y = span[y]$ is not contained in the kernel of an extreme functional from $S_{d(w,1)^*}$ then there exists a continuous metric selection onto $Y.$ Hence to the end of this section we assume that there exists $ f \in ext(S_{d^*(w,1)})$ such that $ Y \subset ker(f).$ By the proof of [\cite{HKL}, Th. 2.2] we have 
\begin{thm}
\label{hkl}
If $ f \in ext(S_{d^*(w,1)})$ then $ f^* =w.$ Moreover if $ \lim_n w(n) =0, $ then $ f \in ext(S_{d^*(w,1)})$ if and only if $ f^* =w.$ 
\end{thm}

We start with a crucial for our considerations lemma.

\begin{lem}
\label{selection}
Assume that $w$ is a strictly decreasing weight function.
Let 
$$ 
y =(y(1),...y(n),...) \in d(w,1) \setminus \{0\} 
$$ 
and let $ Y = span[y] \subset ker(w).$ Assume that there exists $ x=x^* \in d(w,1)$ such that $P_Y(x) = [-1,1]y.$ Set $ z = (z(1),...,z(n),...), $ where 
\begin{equation}
\label{szereg}
z(j) =  \sum_{l=j}^{\infty} |y(l+1) -y(l)|.
\end{equation}
Then $z \in d(w,1),$ (in particular $z(j) \in \mathbb{R}$ for any $ j \in \mathbb{N}$) and 
\begin{equation}
\label{zj}
z(j) \geq |y(j)| \hbox{ for } j \in \mathbb{N}.\end{equation}
Moreover, if 
there exists a subsequence $ (n_k)$ such that 
\begin{equation}
\label{strict}
(y(n_k) - y(n_k +1))(y(n_{k+1}) - y(n_{k+1}+1)<0.
\end{equation}
then there exist compact intervals $ I_1,$ $I_2,$ with $ I_1 \cap I_2 = \emptyset$ and two sequences $ (x^k)$ and $(w^k)$ converging to $z$ such that for any $ k \in \mathbb{N}$ $P_Y(x^k) \subset I_1$ and $ P_Y(w^k) \subset I_2.$
If (\ref{strict}) is not satisfied, $ x=x^* \in d(w,1)$ and $P_Y(x) = [-1,1]y,$ then for any sequences $ (x^k)$ and $(w^k)$ converging to $x$ and any compact intervals $ I_1,$ $I_2,$ satisfying  $P_Y(x^k) \subset I_1$ and $ P_Y(w^k) \subset I_2$ for any $ k \in \mathbb{N},$ we have $ I_1 \cap I_2 \neq \emptyset.$
\end{lem}
\begin{proof}
Since $x=x^*,$  $ \| x\|_{w,1} = \sum_{j=1}^\infty w(j)x(j).$ Since $\sum_{j=1}^{\infty} w(j)y(j)=0,$ and $ P_Y(x) = [-1,1]y,$ for any $ a \in  [-1,1]$
$$
\| x +ay\|_{w,1} = dist(x,Y) = \|x\|_{w,1} = \sum_{j=1}^\infty w(j)x(j) = \sum_{j=1}^\infty w(j)(x(j)+ay(j)).
$$
In particular, since w is strictly decreasing, $ (x+y)^* = x+y$ and $ (x-y)^* = x-y.$ Hence, for any $j \in  \mathbb{N},$ 
\begin{equation}
\label{ineq1}
x(j) \geq |y(j)|,
\end{equation}
and 
\begin{equation}
\label{ineq2}
x(j) - x(j+1) \geq |y(j) - y(j+1)|.
\end{equation}
By (\ref{ineq2}), for $ m < n,$ 
\begin{equation}
\label{ineq3}
x(m) - x(n) \geq \sum_{j=m}^{n-1} |y(j+1) - y(j)|.
\end{equation}
Since $ d(w,1) \subset c_o,$ by (\ref{ineq3}), for any $ m \in \mathbb{N},$
$$
x(m) \geq \sum_{j=m}^{\infty} |y(j+1) - y(j)| =z(m).
$$ 
Since $ x \in d(w,1), $ $z =(z(1),...,)$ defined by (\ref{szereg}) also belongs to $ d(w,1).$
Now we show that
$z(j) \geq |y(j)|$ for $ j \in \mathbb{N}.$ Assume to the contrary that $z(j_o)<|y(j_o)|$ for some $j_o\in\mathbb{N}.$ First suppose that $ |y(j_o)| = y(j_o).$ Since $ z(j) - z(j+1) = |y(j)-y(j+1)|,$
for any $ j > j_o$ we get 
$$
0> z(j_o) -y(j_o) \geq z(j) - y(j).
$$
Hence for $j\geq j_o,$ 
$$
| z(j) -y(j)| \geq |z(j_o) - y(j_o)| > 0.
$$
This shows that $ z-y \notin c_o.$ But $ z-y \in d(w,1) \subset c_o$, so a contradiction. If $ y(j_o)= - |y(j_o)|,$ the proof goes in the same way.
Now assume that there exists a subsequence $ (n_k)$ satisfying (\ref{strict}). Define for $k \in \mathbb{N},$
\begin{equation}
\label{discont}
z^k =(z(1),...,z(k-1),z(k)-2|y(k)-y(k+1)|,z(k+1),...,),
\end{equation}
\begin{equation}
\label{xk}
x^k = z^{n_{2k}}
\end{equation}
and
\begin{equation}
\label{wk}
w^k = z^{n_{2k+1}}
\end{equation}
Without loss of generality we can assume that $ y(n_{2k}) > y(n_{2k}+1).$
Now we claim that $-y \in P_Y(x^k)$ and $ ay \notin P_Y(x^k)$ for $ a >- 1.$
Notice that by (\ref{zj})
$$
x^k(j) + y(j) = z^{n_{2k}}(j) + y(j) \geq 0
$$ 
for $ j \neq n_{2k}.$
By definition of $z,$ for $j < n_{2k}-1$ and $j\geq n_{2k}+1$
$$
x^k(j) + y(j)= z^{n_{2k}}(j) + y(j) \geq z^{n_{2k}}(j+1) + y(j+1) = x^k(j+1)+ y(j+1).
$$ 
Moreover, we have
\begin{align*}
x^k(n_{2k}-1) + y(n_{2k}-1) &= z^{n_{2k}}(n_{2k}-1) + y(n_{2k}-1) = z(n_{2k}-1)+y(n_{2k}-1)\\
&\geq z(n_{2k}) + y(n_{2k}) \geq z^{n_{2k}}(n_{2k}) + y(n_{2k}) \\
&= x^k(n_{2k}) + y(n_{2k}).
\end{align*}
Finally, notice that for any $ a \in [-1,1]$
\begin{align*}
x^{k}(n_{2k}) + ay({n_{2k}})&=z^{n_{2k}}(n_{2k}) + ay(n_{2k})\geq z^{n_{2k}}(n_{2k}+1) + ay(n_{2k}+1)\\
&=x^{k}(n_{2k}+1) + ay(n_{2k}+1)
\end{align*}
if and only if
$$
z(n_{2k}+1) - |y(n_{2k}+1)-y(n_{2k})| + ay(n_{2k}) \geq z(n_{2k}+1) +ay(n_{2k}+1).
$$
Since $ y(n_{2k}) > y(n_{2k}+1)$ the last inequality is satisfied only for $ a=1.$
Hence, 
$$
\| x^{k} + y\|_{w,1}   =   \| z^{n_{2k}} + y\|_{w,1} = \sum_{j=1}^{\infty} (z^{n_{2k}}(j) - (-y(j)))w(j).
$$
Since $ \sum_{j=1}^{\infty}y(j)w(j)=0,$ by Theorem \ref{classical} we get $-y \in P_Y(z^{n_{2k}}),$ as required.
By the above calculations, for any $ a \in [-1,1)$ we conclude
$$ 
x^{k}(n_{2k}) + ay_{n_{2k}} < x^{k}(n_{2k}+1) + ay(n_{2k}+1).
$$ 
Since $w$ is strictly decreasing, for any $a \in [-1,1)$ we have
\begin{align*}
dist(x^k,Y)=&\|x^k +y\|=\sum_{j=1}^{\infty}(x^k(j) + y(j))w(j)=\sum_{j=1}^{\infty}(x^k(j)+ay(j))w(j)\\
<&\sum_{j=1}^{n_{2k}-1}(x^k(j) + ay(j))w(j) + w(n_{2k})(x^{k}(n_{2k}+1) + ay(n_{2k}+1))\\
&+w(n_{2k}+1)(x^{k}(n_{2k})+ay(n_{2k})) + \sum_{j=n_k+2}^{\infty}(x^k(j)+ay(j))w(j)\\
\leq&\|x^k+ay\|_{w,1},
\end{align*}
which shows our claim.
Consequently, for any $ k \in \mathbb{N},$ $ P_Y(x^k)\subset I_1y=[m,-1]y$ for some $ m <-1.$ 
\newline
Since $ y(n_{2k}) > y(n_{2k}+1),$ by (\ref{strict}),  reasoning exactly in the same way, we can show that  for any $ k \in \mathbb{N},$ $ P_Y(w^k) \subset I_2y=[1,p]y$ for some $ p >1.$ By (\ref{c0}), $\|x^k-z\|_{w,1}\rightarrow 0$ and $\|w^k-z\|_{w,1}\rightarrow 0$ as $k\rightarrow\infty$. Since $ I_1 \cap I_2 = \emptyset,$ we get our claim.
\newline
Now assume that there exists $n_o\in\mathbb{N}$ such that either ${y(n)\geq{y(n+1)}}$ for any $n\geq{n_o}$ or $-y(n)\geq{-y(n+1)}$ for any $n\geq{n_o}$. Fix $ x=x^* \in d(w,1)$ with $P_Y(x) = [-1,1]y$
and $(x^n)\subset{d(w,1)}$ with $\norm{x^n-x}{d(w,1)}{}\rightarrow{0}$.  
Without loss of generality, replacing $y$ by $ -y$ if necessary, we can assume that for $ n \geq n_o$ 
\begin{equation}
\label{case1}
y(n)\geq y(n+1).
\end{equation} 
To get our claim, it is enough show that for any compact interval $I\subset\mathbb{R}$ such that $P_Y(x_n)\subset{Iy}$ for every $n\in\mathbb{N},$ $-1 \in I.$
First suppose that 
\begin{equation} 
\label{strict1}
y(n) \neq y(n+1) 
\hbox{ for any }  n \in \mathbb{N}.
\end{equation}
Let  $P_Y(x_n)=[a_n,b_n]y$ where $a_n \leq b_n$. Let $b=\liminf_{n\in\mathbb{N}}\{a_n\}$. We can assume without loss of generality, that $a_n$ converges to $b$. Since $ P_Y(x) = [-1,1]y,$ $|b|\leq 1$.
We show that $b=-1.$ Assume on the contrary that $ b > -1.$ 
Since $a_ny\in{}P_Y(x^n)$, there exists a supporting functional $f^n \in S_{d^*(w,1)}$ such that 
\begin{equation}
\label{equ:**:cs}
f^n(x^n-a_ny)=\norm{x^n-a_ny}{d(w,1)}{}=\dist(x^n,Y)
\end{equation}
and also 
\begin{equation}
\label{equ:4:cs}
\sum_{i=1}^\infty{f^n(i)y(i)}=0.
\end{equation}
By the Banach-Alaoglu Theorem the set $ \{ f^n: n \in \mathbb{N} \}$  has a cluster point $f \in d^*(w,1)$ with respect to the weak$^*$ topology in $d^*(w,1)$ and $ \|f\|_W \leq 1.$
Now assume that there exists a subsequece $ \{n_k\} $ such that for any $k$, $f^{n_k}=f=w.$ 
Hence, since $ w$ is strictly decreasing, for $ k \in \mathbb{N},$ and $ j \in \mathbb{N} $ 
$$
x^{n_k}(j) - a_{n_k}y(j) = (x^{n_k}(j) - a_{n_k}y(j))^*.
$$
Fix $ m \in (-1,b).$ Then  $ a_{n_k} - m >0$ for $k \geq k_o.$  Hence by (\ref{strict1}) for $ j \geq n_o $ and $ k \geq k_o,$ 
$$
(a_{n_k} -m) y(j) > (a_{n_k} -m) y(j+1)
$$
and consequently, 
$$
x^{n_k}(j) - my(j) > (x^{n_k}(j+1) - my(j+1)) >0. 
$$
Since $ m \in (-1,b)$ and $ |b| \leq 1,$ by (\ref{ineq2}), 
$$
x^{n_k}(j) - my(j) > (x^{n_k}(j+1) - my(j+1)) >0. 
$$
for $ k \geq k_1$ and $ j=1,2,...n_o.$
Since $ \sum_{j=1}^{\infty} w(j)y(j) =0,$ by Theorem \ref{classical}, $my \in P_Y(x^{n_k})$ for $ k \geq k_o.$ Since $ m <b, $ we get a contradiction with defintion of $b.$
\newline
So to end the proof under assumption (\ref{strict1}), we construct a subsequence $ \{n_k\}$ such that for any $k$ $ f^{n_k} = f =w.$ Since $ f$ is a cluster point of $ \{ f^n: n \in \mathbb{N} \}$ with respect to the weak$^*$ topology in $d^*(w,1),$ applying the diagonal argument, we can choose a subsequence $ \{n_k\}$ such that $ f^{n_k}(j) \rightarrow f(j)$ for $ j \in \mathbb{N},$ $f^{n_k}(y) \rightarrow f(y)$ and $ f_{n_k}(x) \rightarrow f(x).$
Since $ f^n(y) =0, $ for any $ n \in \mathbb{N},$ $f(y) =0.$ Moreover, since $\| x^n- x\|_{w,1} \rightarrow 0,$ and $ a_n \rightarrow b,$ by (\ref{equ:**:cs})
\begin{align*}
|f^{n_k}(x^{n_k}-a_{n_k}y) - f(x- by)|&\leq |f^{n_k}(x^{n_k}-a_{n_k}y-(x-by))|+|(f-f^{n_k})(x- by)|\\
&\leq \|x^{n_k}-a_{n_k}y-(x-by)\|_{w,1} + |f-f^{n_k}(x)|.
\end{align*}
Consequently, 
$$ 
\lim_k f^{n_k}(x^{n_k}-a_{n_k}y) - f(x- by) =  \lim_k dist(x^{n_k},Y) - f(x-by)=0.
$$ 
Hence
$$
f(x)= f(x-by) = \|x-by\|_{w,1} = dist(x,Y) = \|x\|_{w,1} 
$$
By (\ref{case1}) and \ref{ineq2}), $ x(j) > x(j+1) $ for any $j,$ which means that $x$ has the only one supporting fuctional $w.$ Hence $ f =w.$ By Theorem \ref{singer}, we can assume that there are $g^k$ and $h^k$ extreme functionals of $S_{d(w,1)^*}$ such that $f^{n_k}=\alpha_n g^{k}+(1-\alpha_k)h^{k}$ for some $\alpha_k \in [0,1]$. By Theorem \ref{hkl}, 
$(h^{k})^* = (g^{k})^* =w.$ Since $ w$ is strictly decreasing and  $ f^{n_k}(j) \rightarrow f(j)=w(j)$ for $ j =1,...,n_o,$ $ g^k(j)\rightarrow w(j) $ and $h^k(j) \rightarrow  w(j)$ for $ j=1,...,n_o+1.$ Hence 
$g^{k}(j) = w(j)$ and $h^k(j) = w(j)$ for $j=1,...n_o+1$ and $ k \geq k_o.$ Consequently, for $ k \geq k_o,$ $ f^{n_k}(j) = w(j). $ 
Moreover, since $\sum_{i=1}^{\infty}w(i)y(i)=0$ and by condition (\ref{equ:4:cs}) we conclude for $ k \in \mathbb{N}$
\begin{equation*}
\sum_{i=n_o+1}^{\infty}f^{n_k}(i)y(i)=\sum_{i=n_o+1}^{\infty} w(i)y(i).
\end{equation*}
Notice that $ |f^{n_k}(j)| \leq w(j)$ for any $k.$
By (\ref{case1}) and (\ref{strict1}), we get that $ w(j) = f^{n_k}(j) $ for $ j \geq n_o +1 $ and $ k \in \mathbb{N}$ as required.
\newline
Now assume that (\ref{strict1}) is not satisfied and 
\begin{equation} 
\label{nonzero}
y(n) > 0 \hbox{ for } n \geq n_o.
\end{equation}
By (\ref{ineq1}),
\begin{equation}
\label{general1}
 x(n) = x^*(n) >0 \hbox{ for } n \in \mathbb{N}.
 \end{equation} 
The proof goes in a similar way like under assumption (\ref{strict1}), so we indicate only necessary modifications. Also the same notation will be used.
Ressoning as above, we can choose a subsequence $ \{n_k\}$ such that $ f^{n_k}(j) \rightarrow f(j)$ for $ j \in \mathbb{N},$   $f^{n_k}(y) \rightarrow f(y),$ $ f_{n_k}(x) \rightarrow f(x)$ and 
$$
f(x) = \| x\|_{w,1} = \sum_{j=1}^{\infty} x(j)w(j).
$$
Let $ j_1 = 1 $ and for $ n \geq 2$
$$ 
j_n =\min \{ j > j_{n-1}: x(j_{n-1}) > x(j)\}.
$$ 
Since $ x \in d(w,1) \subset c_o,$ by (\ref{general1}) our definition is correct. Fix $ n_1$ such that $ j_{n_1} > n_o.$
Reasoning like under assumption (\ref{strict1}), we can show that for $k$ sufficiently large $ f^{n_k}(j_l) = w(j_l) $ for $ 2 \leq l \leq n_1.$
Moreover, if $ l \in \{2,...,n_1\}$ is such that 
$$ 
j_l+ 1 < j_{l+1},
$$ 
then for any $j \in \{j_l+1, j_{l+1}-1\} $ and $k$ sufficiently large 
$$
f^{n_k}(j) \in \{ w(j_{l}+1),..., w(j_{l+1}-1)\}.
$$ 
Also if $ l=1,$ 
then $f^{n_k}(j) \in \{ w(1),..., w(j_{2}-1)\}$ for $j=1,...,j_2-1.$
Since 
$$
f^{n_k}(y) = \sum_{j=1}^{\infty} f^{n_k}(j)y(j) = w(y) =\sum_{j=1}^{\infty} w(j)y(j) =0,
$$
by the above reasoning,
$$
\sum_{j=j_{n_1}}^{\infty} f^{n_k}(j)y(j) =\sum_{j=j_{n_1}}^{\infty} w(j)y(j). 
$$
Observe that for any $ j,k \in \mathbb{N},$ $ |f^{n_k}(j) | \leq w(j).$ Hence by (\ref{nonzero}) and (\ref{case1}) $ f^{n_k}(j) = w(j)$ for $ j > j_{n_1}$
Moreover, since for $k$ sufficiently large $ f^{n_k}(j_l)= f(j_l) = w(j_l)$ for $l=1,...,n_1,$ if $j \in \{j_l+1, j_{l+1}-1\} $ then 
$$
f^{n_k}(j) =f(j) \in \{ w(k_{l}+1),..., w(k_{l+1}-1)\}.
$$ 
(compare with the proof under assumption (\ref{strict1})). Observe that, by (\ref{ineq2}), if $ x(j) = x(j+1),$ then $ y(j) = y(j+1).$ Since $ m \in (-1,b),$
$ f^{n_k}(y) =0$ and $ f^{n_k}(x^{n_k} - my) = \|x^{n_k} - my\|_{w,1}$ for $ k$ sufficiently large. By Theorem \ref{classical}, $my \in P_Y(x^{n_k})$ for $ k \geq k_o.$ Since $ m <b, $ we get a contradiction with defintion of $b,$ which finishes the proof under assumption (\ref{nonzero}).
If $ y(n) =0$ for some $n \geq n_o,$ then by (\ref{case1}), $y(j) = 0 $ for $j \geq n.$ Hence for any $ k \in \mathbb{N},$ and $j \geq n$
$$
f^{n_k}(j)(x(j)-my(j)) =f^{n_k}(j)(x(j)-a_{n_k}y(j)).
$$
Hence reasoning as under assumption (\ref{nonzero}) we get that $my \in P_Y(x^{n_k})$ for $ k \geq k_o,$ which completes the proof.
\end{proof}

\begin{thm}
\label{Czeb}
Assume that $w$ is a strictly decreasing weight.
Let 
$$ 
y =(y(1),...y(n),...) \in d(w,1) \setminus \{0\} 
$$ 
and let $ Y = span[y].$ Then $Y$ is not a Chebyshev subspace in $d(w,1)$ if and only if there exists $ M \subset \mathbb{N},$ $card(M)= \infty$ such that $supp(y) \subset M$ and a bijection $ p: \mathbb{N} \rightarrow M,$ $\sigma \in \{ -1,1\}^{\mathbb{N}}$ such that
\begin{equation}
\label{cheb1}
\sum_{j=1}^{\infty} w(j) \sigma(j)y(p(j)) =0,
\sum_{j=1}^{\infty}|\sigma(j)y(p(j))-\sigma(j+1)y(p(j+1))| < \infty
\end{equation}
and 
\begin{equation}
 \label{cheb2}
z=(z(1),z(2),...) \in d(w,1),
\end{equation}
where for $j \in \mathbb{N}$
\begin{equation}
 \label{cheb3}
z(j) = \sum_{l=j}^{\infty} |\sigma(l)y(p(l))-\sigma(l+1)y(p(l+1))|.
\end{equation}
\end{thm}
\begin{proof}
Since $Y$ is not a Chebyshev subspace, there exists $x \in d(w,1)$ such that $ P_Y(x) =[-1,1]y$. By Theorem \ref{classical}, there exists $ f \in S_{d(w,1)^*}$ such that 
$f(x) = \|x\|_{w,1} $ and $ f(y)=0.$ By Lemma \ref{Izbior}, $ f \in ext(S_{d(w,1)^*}).$ By Theorem \ref{hkl}, $ f^* =w.$ Put $M = supp(f).$ Since $w$ is strictly decreasing, $M$ is infinite.
Now we show that $ supp(y) \subset M$ and $ supp(x) \subset M.$  Assume to the contrary that $ y(i) \neq 0$ for some $ i \notin M.$ Then, $ |x(i) - y(i)| > 0$ or  $ |x(i) + y(i)| > 0.$ Assume that  $ |x(i) - y(i)| > 0.$
Since $ x-y \in c_o,$ $ |x(j) -y(j)| < |x(i) -y(i)| $ for $ j \geq j_o.$ Fix $j_1 \in M, $ $j_1 \geq j_o.$ Since $ P_Y(x) =[-1,1]y,$ and $ i \notin M,$
\begin{align*}
\| x-y\|_{w,1} &= f(x-y) = \sum_{j \in M} f(j)x(j)\\
&=\sum_{j=1}^{j_1-1} f(j) (x(j)-y(j)) +  \sum_{j =j_1}^{\infty}  f(j) (x(j)-y(j))\\
&<\sum_{j=1}^{j_1-1}f(j)(x(j)-y(j))+|f(j_1)||x(i)-y(i)|+\sum_{j=j_1+1}^{\infty}f(j)(x(j)-y(j)) \\
&\leq\|x-y\|_{w,1},
\end{align*}
so a contradiction. Since $ f(x) = \|x \|_{w,1},$ reasoning in the same way, we get that $ supp(x) \subset M.$ Since $ supp(y) \subset M,$ and $supp(x) \subset M,$ without loss of generality we can assume that $M = \mathbb{N}.$ Let $p: \mathbb{N} \rightarrow M$ be defined by $ |f(p(j))| = w(j).$ Since $w$ is strictly decreasing and
$f^*=w,$ our defintion is correct. Let $ \sigma(j) = sgn f(p(j)).$  Applying Lemma \ref{selection} to $ y^1 = (\sigma(j)y(p(j)))_{j \in \mathbb{N}}$ and $ x^1 = (\sigma(j)x(p(j)))_{j \in \mathbb{N}}$
we get (\ref{cheb1} - \ref{cheb3}).
\newline
Now assume that (\ref{cheb1} - \ref{cheb3}) are satisfied.
Observe that for any $j \in \mathbb{N},$
$$
z(j) - z(j+1) = |\sigma(j)y(p(j))-\sigma(j+1)y(p(j+1))|.
$$
Hence for any $ j \in \mathbb{N},$ 
$$
z(j) \pm \sigma(j)y(p(j)) \geq z(j+1) \pm \sigma(j+1)y(p(j+1))\geq{0}.
$$
In consequence, by Theorem \ref{classical}, $ P_{Y_1}(z) = [-1,1]y^1$. Notice that a mapping $ 
T: d_{w,1} \rightarrow d_{w,1} $ defined by 
\begin{equation}
\label{isometry}
Tu = (\sigma(j) u(a(j))): j \in \mathbb{N},
\end{equation}
where $ a= p^{-1},$
is a linear, surjective isometry. Define $ x \in d(w,1),$ by $ x(j) = \sigma(j)z(a(j)),$ for $j \in M.$ Observe that $ Tz = x$ and $Ty^1 =y.$ By Lemma \ref{general} applied to $T$
and (\ref{cheb1} - \ref{cheb3}), $ P_Y(x) =[-1,1]y, $ as required. 
\end{proof}
Now we prove the main result of this section.
\begin{thm}
\label{selgen}
Assume that $w$ is a strictly decreasing weight.
Let 
$$ 
y =(y(1),...y(n),...) \in d(w,1) \setminus \{0\} 
$$ 
and let $ Y = span[y].$ Then $Y$ admits a continuous metric selection if and only if $Y$ is a Chebyshev subspace or for any $p$ and $M$ satisfying the requirements of Theorem \ref{Czeb} there exists
$n_o \in \mathbb{N}$ such that for $ n \geq n_o$ 
\begin{equation}
\label{final}
y(p(n)) \geq y(p(n+1))
\end{equation}
or for $ n \geq n_o$
\begin{equation}
\label{final1}
-y(p(n)) \geq -y(p(n+1)).
\end{equation}
\end{thm}
\begin{proof}
If $Y$ is a Chebyshev subspace then by Lemma \ref{Czebyszew}  $Y$ admits a continuous metric selection. If $Y$ is not a Chebyshev subspace, fix $ x \in d{(w,1)}$ such that 
$ P_Y(x) = [a,b]y,$ where $ a <0<b.$  Assume that (\ref{final}) or (\ref{final1}) is satisfied. We show that for arbitrary sequences $w^k \rightarrow x$ and  $x^k \rightarrow x$ if $I_1 $ and $I_2$ are compact intervals
such that $P_Y(w^k) \subset I_1y$ for any $ k \in \mathbb{N}$ and  $P_Y(x^k) \subset I_2y$ for any $ k \in \mathbb{N}$ then $I_1 \cap I_2 \neq \emptyset.$ Without loss of generality we can assume that $ a=-1$ and $b=1.$ Observe that, by Theorem \ref{classical}, there exists $ f \in S_{d^*(w,1)}$ such that 
$f(x) = \|x\|_{w,1} $ and $ f(y)=0.$ By Lemma \ref{Izbior}, $f \in ext(S_{d(w,1)^*}).$ By Theorem \ref{hkl}, $ f^* =w.$ If $f =w,$ then the mapping $p: \mathbb{N} \rightarrow M$ defined in the proof of Theorem \ref{Czeb} is the identity on $ \mathbb{N}.$ Without loss of generality we can assume that (\ref{final}) is satisfied. Hence for $n \geq n_o,$ $y(n) \geq y(n+1). $ By Lemma \ref{selection} $-y \in I_1y \cap I_2y, $ 
which shows our claim. If $ f \neq w,$ we reduce our proof to the case $f=w.$ To do that, put $ M = supp(f).$ Since $ w$ is strictly decreasing, $M$ is infinite. Reasoning as in the proof of Theorem \ref{Czeb} we can show that $ supp(y) \subset M$ and 
$ supp(x) \subset M.$  Since $f^*=w$ it follows that $M = \mathbb{N}.$ Let $p: \mathbb{N} \rightarrow M$ be defined by $ |f(p(j))| = w(j)$ and let $ \sigma(j) = sgn f(p(j)).$  (Compare with the proof of Theorem \ref{Czeb}). Define $ y^1 = (\sigma(j)y(p(j))_{j \in \mathbb{N}},$ $ Y^1 = span[y^1]$ and $ x^1 \in d(w,1),$ by $ x^1(j) = \sigma(j)x(p(j)),$ for $j \in M,$ where $ a= p^{-1}.$ By Lemma \ref{selection} and by (\ref{final}) applied to 
$x^1$ and $ y^1$ for any sequence $z^k \in d(w,1)$ converging to $x^1$ and a compact interval $ I$ such that $ P_{Y^1}(z^k) \subset I y^1,$ for any $ k \in \mathbb{N},$ $ -1 \in I. $ Observe that 
$Tx^1=x$ and $Ty^1=y,$ where $T$ is defined by (\ref{isometry}). 
By Lemma \ref{general}, apllied to $T$, $x^1$ and $ y^1$ for any sequence $x^k \in d(w,1)$ converging to $x$ and a compact interval $ I$ such that $ P_{Y}(x^k) \subset I y^1,$ for any $ k \in \mathbb{N},$ we get $ -1 \in I, $ which gives our claim. 
\newline
To prove the converse, assume that there are $p$ and $M$ satisfying the requirements of Theorem \ref{Czeb} such that (\ref{final}) and (\ref{final1}) are not satisfied. First suppose that $ p = id_{\mathbb{N}}.$ 
This implies that $y$ satisfies (\ref{strict}) for some subsequence $ \{ n_k\}.$ Let $ z \in d(w,1)$ be defined by (\ref{cheb2}) and (\ref{cheb3}). By the proof of Lemma \ref{selection} there exist two sequences $ x^k$ and $w^k$
converging to $z$ such that there exist compact intervals $ I_1,$ $ I_2,$ with $ I_1 \cap I_2 = \emptyset$ having the property that for any $ k \in \mathbb{N}$ $P_Y(x^k) \subset I_1$ and $ P_Y(w^k) \subset I_2.$
By Theorem \ref{deutsch}, $Y$ does not admit the continuous metric selection.
\newline
Applying the isometry $T$ defined by (\ref{isometry}) and Lemma \ref{general}, we can reduce the proof in general case to the above reasoning, which gives our claim.
\end{proof}
Now we show four examples illustrating possible applications of Theorem \ref{selgen}.
\begin{exam}
Let $ 1= w(1) > w(2)> ..$ be a strictly decreasing weight such that $ \sum_{n=1}^{\infty}w(n) = +\infty .$ Let $ y=(1, y(2),0,0...)$ be so chosen $ y(2) <0$  and $ y(2) \neq \frac{-w(i)}{w(j)}$ for any $ i,j \in \mathbb{N},$ $i \neq j.$ Since $ \mathbb{R} $ is not countable, such an $y(2)$ exists. By Theorem \ref{hkl} for any $f \in ext(S(d^*(w,1)))$, $f(y) \neq 0.$ Hence, by Lemma \ref{Izbior}  
any $ x \in d(w,1)$ possesses strongly unique best approximation in $ Y = span[y].$ By the Freud Theorem (see \cite{Chen}, p. 82) the projection operator (which is single-valueds in this case) satisfies the local Lipschitz condition.
\end{exam}
\begin{exam}
 Let $w(1)=1$, $ w(n) = \frac{1}{n} ,$ $y(n) = \frac{1}{n},$ for $n \geq 2$  and $ y(1) = - \sum_{n=2}^{\infty}y(n)w(n).$  Observe that in this case $ \sum_{n=1}^{\infty} |y(n+1) - y(n)| < \infty . $
By Theorem \ref{Czeb}, $Y$ is not a Chebyshev subspace of $d(w,1).$ By Theorem \ref{selgen},  $Y$ admits a continuous metric selection. 
\end{exam}
\begin{exam}\label{no-cont-sele}
Let $ w(n) = \frac{1}{n} $ and for $ k \geq 1$ $y(2k) = \frac{1}{(2k+1)^2},$ $y(2k+1) = \frac{1}{(2k)^2}$  and $ y(1) = - \sum_{n=2}^{\infty}y(n)w(n).$  Observe that in this case $ \sum_{n=1}^{\infty} |y(n+1) - y(n)| < \infty . $
By Theorem \ref{Czeb} $Y$ is not a Chebyshev subspace of $d(w,1).$ By Theorem \ref{selgen}  $Y$ does not admit a continuous metric selection. 
\end{exam}

\begin{exam}
Now we consider Example \ref{no-cont-sele} with $y(1)=-\sum_{n=2}^{\infty}y(p(n))w(n)$, where $p:\mathbb{N}\rightarrow\mathbb{N}$ is a function given by $p(2k)=2k+1$ and $p(2k+1)=2k$ for any $k\in\mathbb{N}$. Notice that $\sum_{n=1}^{\infty}|y(p(n+1))-y(p(n))|<\infty.$
By Theorem \ref{Czeb} $Y$ is not a Chebyshev subspace of $d(w,1).$ On the other hand, by Theorem \ref{selgen} $Y$ admits a continuous metric selection. 
\end{exam}

$\begin{array}{lr}
\textnormal{\small Maciej CIESIELSKI} & \textnormal{\small Grzegorz Lewicki}\\
\textnormal{\small Institute of Mathematics} & \textnormal{\small Department of Mathematics and Computer Science}\\
\textnormal{\small Pozna\'{n} University of Technology} & \textnormal{\small Jagiellonian University}\\
\textnormal{\small Piotrowo 3A, 60-965 Pozna\'{n}, Poland} & \textnormal{\small \L ojasiewicza 6, 30-348 Krak\'ow, Poland}\\
\textnormal{\small email: maciej.ciesielski@put.poznan.pl;} & \textnormal{\small email: grzegorz.lewicki@im.uj.edu.pl}
\end{array}$

\end{document}